\pdfoutput=1
\documentclass[11pt]{article}

\usepackage[utf8]{inputenc}
\def\final{1}
\usepackage{color}
\usepackage{subfigure}
\usepackage{url}
\usepackage{graphicx}
\usepackage{amsmath}
\usepackage{amsthm}
\usepackage{bm}
\usepackage[plainpages=false]{hyperref}
\usepackage{amscd}

\usepackage{titlesec}
\titleformat{\section}
{\normalfont\Large\bfseries\filcenter}{\thesection.}{1 ex}{}
\titleformat{\subsection}
{\normalfont\normalsize\bfseries}{\thesubsection}{1 ex}{}
\titleformat{\subsubsection}[runin]
{\normalfont\normalsize\bfseries}{\thesubsubsection.}{1 ex}{}

\usepackage[charter]{mathdesign}
\usepackage[mathcal]{eucal}

\usepackage[margin=1.2 in]{geometry}

\usepackage[square,numbers,sort&compress]{natbib}

\ifnum\final=0
\newcommand{\mynote}[1]{\marginpar{\tiny\sf #1}}
\else
\newcommand{\mynote}[1]{}
\fi

\usepackage{thmtools,thm-restate}
\declaretheorem{theorem}
\declaretheorem[sibling=theorem,name=Corollary]{cor}
\declaretheorem[within=section]{lemma}
\declaretheorem[sibling=lemma,name=Proposition]{prop}
\declaretheorem[style=remark,sibling=lemma]{remark}
\declaretheorem[style=remark]{definition}
\declaretheorem[name=Problem]{problem}
\declaretheorem[name=Question]{question}

\newcommand{\figref}[1]{Figure \ref{fig:#1}}

\newcommand{\lemref}[1]{Lemma \ref{lemma:#1}}

\newcommand{\corref}[1]{Corollary \ref{cor:#1}}
\newcommand{\propref}[1]{Proposition \ref{prop:#1}}
\newcommand{\theoref}[1]{Theorem \ref{theo:#1}}
\newcommand{\theorefX}[1]{\ref{theo:#1}}
\newcommand{\secref}[1]{Section \ref{sec:#1}}
\renewcommand{\eqref}[1]{({\ref{eq:#1}})}

\newcommand{\lemlab}[1]{\label{lemma:#1}}
\newcommand{\proplab}[1]{\label{prop:#1}}
\newcommand{\theolab}[1]{\label{theo:#1}}
\newcommand{\corlab}[1]{\label{cor:#1}}
\newcommand{\seclab}[1]{\label{sec:#1}}
\newcommand{\eqlab}[1]{\label{eq:#1}}

\renewcommand{\vec}[1]{\mathbf{#1}}

\newcommand{\elltilde}{\tilde{\bm{\ell}}}

\newcommand{\bgamma}{\bm{\gamma}}

\newcommand{\bomega}{\bm{\omega}}

\newcommand{\Euc}{\operatorname{Euc}}

\newcommand{\Id}{\operatorname{Id}}

\newcommand{\rk}{\operatorname{rk}}

\newcommand{\Mat}{\operatorname{Mat}}
\newcommand{\pr}{\operatorname{pr}}
\newcommand{\F}{\mathbb{F}}

\newcommand{\eop}{\hfill$\qed$}
\newcommand{\mc}{\mathcal}
\newcommand{\pcross}{\uparrow}
\newcommand{\ncross}{\downarrow}
\newcommand{\Gamd}[1]{$\Gamma$-$(#1, #1)$}
\newcommand{\tdeg}{\widetilde{\deg}}

\newcommand{\C}{\mathbb{C}}
\newcommand{\R}{\mathbb{R}}
\newcommand{\Q}{\mathbb{Q}}
\newcommand{\Z}{\mathbb{Z}}

\renewcommand{\div}{\operatorname{div}}

\DeclareMathOperator{\Galois}{Gal}
\DeclareMathOperator{\HH}{H}
\DeclareMathOperator{\tr}{tr}
\DeclareMathOperator{\Aut}{Aut}
\DeclareMathOperator{\Func}{Func}
\DeclareMathOperator{\Hom}{Hom}
\DeclareMathOperator{\SL}{SL}
\DeclareMathOperator{\GL}{GL}
\DeclareMathOperator{\ssl}{sl}
\DeclareMathOperator{\vol}{vol}

\begin{document}
\title{Ultrarigid periodic frameworks}
\author{Justin Malestein\thanks{Einstein Institute of Mathematics, Hebrew University, Jerusalem, \url{justinmalestein@math.huji.ac.il}}
\and Louis Theran\thanks{Aalto Science Institute and Department of Computer Science,
Aalto University, \url{louis.theran@aalto.fi}}}
\date{}
\maketitle
\begin{abstract}
\begin{normalsize}
We give an algebraic characterization of when a $d$-dimensional periodic framework has no
non-trivial, symmetry preserving, motion for \emph{any} choice of
periodicity lattice. Our condition is decidable, and we provide
a simple algorithm that does not require complicated algebraic
computations.  In dimension $d=2$, we give a combinatorial characterization in the
special case when the  the number of edge orbits is the minimum possible for ultrarigidity.
All our results apply to a fully flexible, fixed area, or fixed periodicity lattice.
\end{normalsize}
\end{abstract}

\section{Introduction} \seclab{intro}
A \emph{periodic framework} is an infinite structure in Euclidean $d$-space, made of fixed-length bars
connected by universal joints and
symmetric with respect to a lattice $\Gamma$.  To fully describe the model,
we need to describe the allowed motions. The Borcea-Streinu deformation
theory \cite{BS10}, by-now the standard in the mathematical
literature on periodic frameworks, allows precisely those motions which preserve the
lengths and  connectivity of the bars and symmetry with respect to $\Gamma$, but
not the geometric representation of $\Gamma$, which is allowed to deform continuously.
We give more detail shortly, in \secref{setup}, but want to call out here
the key features of \emph{forced symmetry}  and \emph{deformable lattice representation}.

For this setting there are good algebraic \cite{BS10} and, in dimension $2$,
combinatorial \cite{MT13} characterizations of rigidity and flexibility.
Simply dropping the symmetry forcing altogether is known to lead to quite complicated
behavior \cite{OP09}, and the tools from \cite{BS10,MT13} do not apply directly.
One alternative approach is to study the behavior when relaxing
the symmetry constraints along a decreasing sequence of sublattices.  In this paper, we
will consider the extreme case, characterizing the periodic frameworks that
are infinitesimally rigid and remain so when the symmetry constraint is relaxed to
\emph{any} sublattice.

\subsection{The basic setup and background}\seclab{setup}
A periodic framework is defined by the triple
$(\tilde{G},\varphi,\elltilde)$, where $\tilde{G}$ is an infinite graph, $\varphi: \Z^d \to \Aut(G)$ is a free
$\Z^d$-action with finite quotient, and
$\elltilde : E(\tilde{G})\to \R_{> 0}$ is a $\varphi$-equivariant function assigning a length to each edge.
A \emph{realization} $(\vec p,\vec L)$ of $(\tilde{G},\varphi,\elltilde)$ is given by a
function $\vec p : V(\tilde{G}) \to \R^d$ and a matrix $\vec L\in \R^{d\times d}$ such that $\vec p$
is equivariant with respect to the lattice generated by the columns of $\vec L$, i.e.,
$$\begin{array}{rcll}
||\vec p(j) - \vec p(i)||^2 & = &\elltilde(ij)^2 & \text{for all $ij\in E(\tilde{G})$} \\
\vec p(\varphi(\gamma)(i)) & = &\vec p(i) + \vec L\cdot \gamma & \text{for all $i\in V(\tilde{G})$ and $\gamma\in \Z^d$}
\end{array}$$
Realizations are denoted by $\tilde{G}(\vec p,\vec L)$.
The set of all realizations is denoted $\mathcal{R}(\tilde{G},\varphi,\elltilde)$, and the \emph{configuration space}
$\mathcal{C}(\tilde{G},\varphi,\elltilde) = \mathcal{R}(\tilde{G},\varphi,\elltilde)/\Euc(d)$ is then defined as the
quotient of the realization space by Euclidean isometries.  A realization  $\tilde{G}(\vec p,\vec L)$ is then \emph{rigid}
if it is isolated in the configuration space and otherwise \emph{flexible}. The realization $\tilde{G}(\vec p,\vec L)$
is \emph{infinitesimally rigid} if the tangent space at $(\vec p, \vec L)$ in $\mathcal{R}(\tilde{G},\varphi,\elltilde)$ is $d+1\choose2$-dimensional and otherwise
\emph{infinitesimally flexible}.

The essential results on this model from \cite{BS10}, which introduced it, are that: (i) the realization and
configuration spaces are \emph{finite-dimensional algebraic varieties}; (ii) generically, rigidity and
flexibility are determined completely by the absence or presence of a \emph{non-trivial infinitesimal flex}, which can
be tested for in polynomial time via linear algebra;
(iii) generic rigidity and flexibility are properties of the finite \emph{colored quotient graph} of
$(G,\bgamma)$, which is a finite directed graph, with its edges labeled by elements of $\Z^d$.
(See \secref{infmot} for the dictionary between infinite periodic graphs and colored graphs.)

In dimension two, \cite[Theorem A]{MT13}, gives a combinatorial characterization of generic periodic
rigidity, in terms of the colored quotient graph.  The characterization is a good one, in the sense that it is
decidable by polynomial-time, combinatorial algorithms.  For higher dimensions, as is also the case for
finite bar-joint frameworks, finding a similar combinatorial characterization is a notable open problem.

All of the above-mentioned results on periodic frameworks rely, in an essential way, on symmetry-forcing.
Simply dropping the symmetry requirements for the allowed motions leads to configuration spaces that are not
treatable via the techniques from \cite{BS10}.  Additionally, starting with
a rigid periodic framework $\tilde{G}(\vec p,\vec L)$ and relaxing the symmetry
constraint to \emph{any sublattice at all} produces a framework that is,
a priori, non-generic, and so we cannot naively apply the results of \cite{BS10,MT13} to it.

\subsection{Ultrarigidity}
We define a periodic framework $\tilde{G}(\vec p,\vec L)$ to be \emph{periodically ultrarigid}
(simply \emph{ultrarigid}, for short, since there is no chance of confusion)
if it is rigid and remains so after relaxing the symmetry constraint to \emph{any} sublattice.
This definition and terminology are from \cite{BS12}.
That not all infinitesimally rigid periodic frameworks are ultrarigid was observed in \cite{BS10}.
The question of which colored graphs are
generically ultrarigid was raised, for dimension $2$, in \cite{T11}\footnote{See, in particular, the slide
\url{http://www.fields.utoronto.ca/audio/11-12/wksp_symmetry/theran/index.html?42;large\#slideloc}, and
the discussion in  \cite[Section 19.5]{MT13}.}
under the name ``sublattice question''. A similar question for periodic frameworks in all dimensions
was raised in \cite[Question 8.2.7]{R11}.

For any sublattice $\Lambda < \Z^d$, one can compute an associated rigidity matrix whose kernel
is the space of infinitesimal motions periodic relative to $\Lambda$.
However, this does not provide a formulation that immediately provides a finite
certificate of infinitesimal ultrarigidity. One must, a priori, compute the rank of infinitely many
matrices. (A finite certificate of infinitesimal ``ultraflexibility''
is given simply by the rigidity matrix associated with a particular sublattice that yields a non-trivial
infinitesimal motion.)

\subsection{Results and roadmap}
Our main theorem is an effective algebraic characterization of infinitesimal ultrarigidity. To state it, we first recall
that a \emph{torsion point} in $(\C^\times)^d = (\C \setminus \{0\})^d$ is any point $\omega = (\zeta_1, \dots, \zeta_d) \in \C^d$ where $\zeta_1, \dots, \zeta_d$
are roots of unity. Equivalently, a torsion point is any point with finite order in the group $(\C^\times)^d$ where the group operation
is component-wise multiplication. Let $\vec 1 = (1, 1, \dots, 1)$

\begin{restatable}{theorem}{mainthm}\theolab{main}
Let $\tilde{G}(\vec p,\vec L)$ be an infinitesimally rigid periodic framework in dimension $d$, with
colored quotient graph $(G,\bgamma)$.  Then there is an explicit constant $N(d,G, \vec p, \vec L)$ depending only on
$d, G, \vec p, \vec L$, and a finite collection of polynomials $p_1, p_2,\ldots, p_k \in \mathbb{C}[x_1,x_2,\ldots,x_d]$
such that $\tilde{G}(\vec p,\vec L)$ is infinitesimally ultrarigid if and only if $\vec 1$ is the only torsion point
of order $\leq N(d,G, \vec p, \vec L)$ in the common solution set of $p_1, p_2,\ldots, p_k$. If $\vec p, \vec L$ are rational, we can replace
$c(d,G, \vec p, \vec L)$ with a constant $c(d,G)$ depending only on $d, G$.
\end{restatable}
\theoref{main} follows directly from:
a characterization of infinitesimal ultrarigidity from which
the polynomials $p_i$ are derived (\theoref{ultramatrixproj} below);
and a general theorem regarding torsion points as common solutions
to polynomials of bounded degree (\theoref{torsionpoints} below).

\subsubsection{Infinitesimal ultrarigidity}
The polynomials $p_i$ are given by minors of a matrix $\hat S_{G,\vec d}$ with entries
in the group ring $\R[\Z^d]$ with the pattern
\begin{equation}\eqlab{ultra-matrix}
\bordermatrix{
&        &  i           &       & j                                 &           \cr
&  \dots &  \dots      &  \dots  &      \dots                       &   \dots   \cr
ij &  \dots & -\vec d_{ij} & \dots & \vec d_{ij} \otimes \gamma_{ij} & \dots     \cr
&  \dots &  \dots      &  \dots  &      \dots                       &   \dots
}
\end{equation}
where $\vec d_{ij} = \vec p_j - \vec p_i + \vec L\cdot \vec \gamma_{ij}$ is the edge vector
associated with the colored edge $ij\in E(G)$,  $\gamma_{ij}$ is viewed as an element
of the group ring, and $\otimes$ denotes component-wise multiplication.

We can view $\hat S_{G,\vec p,\vec L}$ as a matrix with monomial entries in
$\R[x_1, x_1^{-1}, \dots, x_d, x_d^{-1}]$ via the natural isomorphism
$\R[\Z^d] \to \R[x_1, x_1^{-1}, \dots, x_d, x_d^{-1}]$ defined by
$\gamma \mapsto x^\gamma := x_1^{\gamma_1} x_2^{\gamma_2} \cdots x_d^{\gamma_d}$ where
$\gamma_i$ are the components of $\gamma$. That this \emph{rigidity matrix}
captures the infinitesimal motions is the first part of the proof of \theoref{main},
and it follows from\footnote{Previous versions of this paper omitted the reference
to \cite{CSS14}, of which we were unaware.  We regret the error.}:
\begin{restatable}[name={\cite{CSS14,PowerRUM}}]{theorem}{ultramatrixproj}\theolab{ultramatrixproj}
Let $\tilde{G}(\vec p,\vec L)$ be an infinitesimally rigid periodic
framework in dimension $d$, with colored quotient graph $(G,\bgamma)$ on $n$ vertices.
Then, $\tilde{G}(\vec p,\vec L)$ is infinitesimally ultrarigid
if and only if for every torsion point $\bomega \neq \vec 1$, evaluating the entries of
$\hat S_{G,\vec d}$ at $\bomega$ results in a matrix of rank $dn$.
\end{restatable}
In \secref{infmot} we provide a direct derivation of \theoref{ultramatrixproj}, since
this form of  $\hat S_{G,\vec d}$
gives exactly  the polynomials $p_i$ appearing in \theoref{main}.
However, one may deduce it from previous work as follows.  The rigidity matrix
$\hat S_{G,\vec d}$ is a simple, rank-preserving,
transformation, of a rigidity matrix from
\cite{PowerRUM}. The key difference between our setting and that of
\cite{PowerRUM} is that we do not start with the assumption that all infinitesimal
motions must fix the lattice representation. To bridge this gap, we can then
use \cite[Theorem 5.1]{CSS14}.  Translated to our terminology, \cite[Theorem 5.1]{CSS14}
says that if
$\tilde{G}(\vec p,\vec L)$ is not infinitesimally ultrarigid, then there is a sublattice
$\Lambda' < \Lambda$ such that there is a non-trivial infinitesimal motion
$(\vec v,\Id)$, periodic with respect to $\Lambda'$.
This brings the question back into the setting of \cite{PowerRUM}, and
\theoref{ultramatrixproj} follows.

\subsubsection{Torsion points}
\theoref{main} states that
checking finitely many possibilities is sufficient to ensure $\vec 1$ is the only torsion point in
the variety defined by the minors of the above rigidity matrix.  This is a consequence of a
more general result, which is a consequence of the more explicit \corref{esttwo}
in \secref{algo}.
\begin{restatable}{theorem}{torsionpoints}\theolab{torsionpoints}
For any collection of polynomials $p_1,\ldots,p_k\in \C[x_1^{\pm 1},\ldots,x_d^{\pm 1}]$,
there is a number $N_0$, depending only on the degrees of the $p_i$ and the coefficient field,
such that if $\vec 1$ is the only torsion point up to order $N_0$
in the common solution set of $p_1,\ldots,p_k$,
then $\vec 1$ is the only torsion point in the common solution
set of $p_1,\ldots,p_k$.
\end{restatable}
A number of similar statements are known.  Hindry \cite[Theorem 1]{H88} gives
an effective upper bound on the minimal order of torsion points in the case where the
$p_i$ are defined over a number field.  Bombieri and Zannier \cite{BZ95} bound the
minimal order of torsion points in terms of degrees of the $p_i$
and the heights of coefficients.  We do not, however, know any result
that implies exactly the statement of \theoref{torsionpoints}.

\subsubsection{Algorithmic results}
In \secref{algo}, we provide an explicit $N_0$ suitable for \theoref{torsionpoints} which depends on the degrees and coefficient
fields of the $p_i$. Consequently, we obtain:
\begin{cor}\corlab{decidable}
Infinitesimal ultrarigidity is a decidable property.
\end{cor}
Apart from our own \theoref{torsionpoints}, this also follows from the combination of \theoref{ultramatrixproj} and
the existence of known algorithms computing the torsion cosets lying in an algebraic variety (e.g. \cite{AS12, Leroux}).
For periodic frameworks with rational coordinates, we give a more efficient
algorithm. Here $\| \cdot \|_1$ denotes the $L_1$-norm of a vector.
\begin{restatable}{theorem}{twodalgo}\theolab{twodalgo}
Let $\tilde{G}(\vec p,\vec L)$ be an infinitesimally rigid periodic framework with
$\vec p$ and $\vec L$ rational, and let $(G,\bgamma)$ be the associated colored graph with $n$ vertices and $m$ edges and
$D = \sum_{ij \in E(G)} \|\gamma_{ij}\|_1$.
There is an algorithm with running time polynomial in $m$, $n$, and $D$
that decides the infinitesimal ultrarigidity of $\tilde{G}(\vec p,\vec L)$.
\end{restatable}
The algorithm is presented and analyzed in \secref{thealgorithm}.  The
algorithm is not polynomial time, because of the dependence on $D$, though
in many applications we will have $D = O(m)$.  Additionally, the
implied constants grow exponentially in the ambient dimension $d$ and
the exponents of $m$, $n$, and $D$ in the running time are $\Theta(d^2)$.

Note that a kind of finiteness result \cite[Corollary 6.1, 6.2]{CSS14} is proved by
Connelly--Shen--Smith. However, the results are considerably different,
and, e.g., are not suitable for producing an algorithm to check infinitesimal ultrarigidity.

\subsubsection{Combinatorial results}
For $d=2$ we are also able to give a combinatorial characterization in the special case where
the quotient $(G, \bgamma)$ is a graph on $n$ vertices and $m = 2n+1$ edges.  The families of $\Delta$-$(2,2)$ and
colored-Laman graphs appearing in the statement of \theoref{dtwocomb} come from \cite{MT13,MT13b}
and are defined in \secref{graphmats}.
\begin{restatable}{theorem}{dtwocomb}\theolab{dtwocomb}
Let $\tilde{G}(\vec p,\vec L)$ be a generic $2$-dimensional periodic framework with
associated colored graph $(G,\bgamma)$ on $n$ vertices and $m=2n+1$ edges.  Then $\tilde{G}(\vec p,\vec L)$
is infinitesimally ultrarigid if and only if $(G, \bgamma)$ is colored-Laman and $(G, \Psi(\bgamma))$ is $\Delta$-$(2,2)$ spanning for all finite cyclic
groups $\Delta$ and epimorphisms $\Psi: \Z^2 \to \Delta$. Moreover, it is sufficient to check a finite set of epimorphisms $\Psi$ which
depends only on $(G, \bgamma)$.
\end{restatable}
For the above theorem, \emph{generic} means that the coordinates of $\vec p(i)$ and $\vec L$ are
algebraically independent over $\Q$, for a choice
of vertex representatives $i \in V(\tilde G)$. Consequently, graphs satisfying the above combinatorial
conditions have a full measure set of ultrarigid frameworks. At present, we are unable to say whether the
set of infinitesimal ultrarigid frameworks
contains an open dense set of all periodic realizations. However, \theoref{main} implies that
among {\it rational} realizations,
the infinitesimally ultrarigid ones are the complement of a proper algebraic variety. We also remark
that it is unclear whether, even generically,
infinitesimal ultrarigidity must coincide with ultrarigidity. (It is untrue in the case of the
fixed lattice.) We discuss these issues in more detail in \secref{conc}.

\paragraph{Fixed lattice and fixed volume}
Aside from \theoref{dtwocomb}, all of the above theorems transfer straightforwardly to ultrarigidity in the context
of a fixed lattice (\emph{f.l.}, for short) or lattices of fixed volume (\emph{f.v.} in $d\ge 3$, or \emph{f.a.}
in $d=2$, for short).
Moreover, with a few additional lemmas we can also prove fixed-lattice and fixed-area analogues of
\theoref{dtwocomb}. The unit-area-Laman graphs and Ross graphs are defined below in \secref{graphmats}.
\begin{restatable}{theorem}{dtwocombflfa}\theolab{dtwocombflfa}
Let $\tilde{G}(\vec p,\vec L)$ be a generic $2$-dimensional periodic framework with
associated colored graph $(G,\bgamma)$ on $n$ vertices and $m=2n$ edges.  The
following are equivalent:
\begin{itemize}
\item[(i)] $\tilde{G}(\vec p,\vec L)$ is infinitesimally f.l. ultrarigid
\item[(ii)] $\tilde{G}(\vec p,\vec L)$ is infinitesimally f.a ultrarigid
\item[(iii)] $(G,\bgamma)$ is unit-area-Laman and $(G,\Psi(\bgamma))$ is $\Delta$-$(2,2)$ spanning for
all finite cyclic groups $\Delta$ and epimorphisms $\Psi : \Z^2 \to \Delta$.
\item[(iv)] $(G,\bgamma)$ is Ross-spanning and colored-Laman-sparse,
and $(G,\Psi(\bgamma))$ is $\Delta$-$(2,2)$ spanning for
all finite cyclic groups $\Delta$ and epimorphisms $\Psi : \Z^2 \to \Delta$.
\end{itemize}
\end{restatable}
We note that unit-area-Laman graphs are always generically rigid in the fixed-lattice model.
The combinatorial conditions in \theoref{dtwocomb} are equivalent to ones that do not
reference any finite quotients of $\Gamma$ (see \lemref{altform} below).  This is useful
for computational purposes, and the conditions in Theorems \theorefX{dtwocomb} and
\theorefX{dtwocombflfa} are all checkable in polynomial time. \secref{combalg} gives the
algorithms.

\subsection{Motivations}
Infinite frameworks have been used as geometric models for crystalline
structures (e.g., \cite{SwainDove}) for quite some time.  A specific class of silicates, zeolites,
which exhibit flexibility \cite{SWTT06} has been studied via bar-joint framework models
quite a bit in the recent past \cite{R06,KDRT11}.  Studies from physics and engineering
have used a variety of ad-hoc deformation theories for infinite frameworks.

Of particular interest here are perhaps the recent study \cite{SSML12}
of the Kagome lattice, which observes the emergence of long range phonons in a particular
very symmetric realization, while observing that in other realizations, the floppy modes
that emerge appear to be determined by the lattice's topology.  The response letter \cite{V12}
points to the role of geometry in such special configurations.

\subsection{Other related work}
Our method is based on the representation theory of $\Z^d$.  The use of
representation theory to study frameworks originates, to our knowledge,
with \cite{FG00}.  For finite discrete subgroups of $\Euc(d)$, the
analog of ultrarigidity is ``incidental symmetry'' (see, e.g., \cite{S10,S10a,ST13}).

A nontrivial class of ultrarigid and f.a. ultrarigid examples constructed from periodic
pointed pseudo-triangulations are described in
\cite{BS15}\footnote{The reference \cite{BS15} has appeared as an extended abstract in \cite{BS14}.}. Some implications and related questions are discussed in \secref{conc}.

\subsection{Acknowledgements}
LT is supported by the European Research Council under the European Union’s Seventh Framework
Programme (FP7/2007-2013) / ERC grant agreement no 247029-SDModels. JM is supported by the European
Research Council under the European Union’s Seventh Framework Programme (FP7/2007-2013) /
ERC grant agreement no 226135.

\section{Rigidity matrices} \seclab{infmot}
In this section, we characterize infinitesimal ultrarigidity of periodic frameworks in terms of matrices.
Ultrarigidity turns out to be characterized in a natural way by an $\R[\Z^d]$-linear system and,
concretely, by one $\R[\Z^d]$ matrix and one real matrix. We now fix some dimension $d$ and let
$\Gamma = \Z^d$. In this section, we will write the group operation in  $\Gamma$ multiplicatively.

The matrices $\hat S$ appearing
below are essentially the matrices $\phi_C(z)$ defined by Power in \cite{PowerRUM} where the connection
to ultrarigidity is also made. We present a different derivation of them by starting from motions
periodic with respect to some finite index $\Lambda < \Z^d$ and using representation theory.
Moreover, in \cite{PowerRUM}, Power only discusses motions not deforming the lattice
representation or ``unit-cell'' while the derivation here starts without that assumption.
As discussed in the introduction, an alernative path is to reduce the more general
question to the setting of \cite{PowerRUM} via a result from \cite{CSS14}.

\paragraph{Colored quotient graph}
A periodic graph is a pair $(\tilde G, \varphi)$ where $G$ is an infinite graph and $\varphi: \Gamma \to \Aut(\tilde G)$ is a
free action of $\Gamma$ on $\tilde G$. We will assume that the number of vertex and edge orbits is finite.
Since $\Gamma$ acts freely, the quotient map $\tilde G \to \tilde G/\Gamma$ is a covering map, and the data $(\tilde G, \varphi)$ can be
encoded by $\tilde G/\Gamma$ and a representation $\pi_1(\tilde G/\Gamma, i) \to \Gamma$.
A more convenient encoding is via colors (or ``gains''). Let $G= \tilde G/\Gamma$, and choose some orientation of the edges.
For each vertex $i \in V(G)$, choose a representative vertex $\tilde i \in V(\tilde G)$ of the corresponding orbit. Given any
edge $ij$, there is a unique lift to $E(\tilde G)$ with head $\tilde i$; the tail is $\gamma_{ij} \cdot \tilde j$ for a unique $\gamma_{ij} \in \Gamma$,
and $\gamma_{ij}$ is the color for $ij$. In general, a $\Gamma$-colored graph $(G, \bgamma)$ (for arbitrary groups $\Gamma$) is a directed
graph with edges labelled by elements of $\Gamma$. (These are also known as ``gain graphs''.)

Using our choice of representatives, we can furthermore identify $V(G) \times \Gamma \cong V(\tilde G)$ via $(i, \gamma) \mapsto \gamma \cdot \tilde i$. For any edge $ij \in E(G)$,
there is a corresponding orbit of edges where $(i, \gamma)$ is connected to $(j, \gamma \gamma_{ij})$ for all $\gamma \in \Gamma$.

\subsection{Parameterizing periodic realizations}
A ($\Gamma$-periodic) realization of $(\tilde G, \varphi)$
is an equivariant pair $(\vec p, \vec L)$ of a function $\vec p: V(\tilde G) \to \R^d$ and a representation $\vec L: \Gamma \to \R^d$ where ``equivariant'' means
that $\vec p(\gamma \cdot i) = \vec p(i) + \vec L(\gamma)$. Using the free action, we will describe this in slightly different language, and
then give an alternate parameterization.

First, set $\mc X = \Func(\Gamma, \R)$ which has a natural (left/right\footnote{Since $\Gamma$ is abelian, there is no distinction between
left and right actions. These formalisms describing the infinitesimal motions should generalize to crystallographic groups, so we have endeavored
to rely as little as possible on this fact and to use formulas which generalize more easily.}) action, namely
$(\gamma \cdot f)(\gamma_0) = f(\gamma^{-1} \gamma_0)$. Then any (not necessarily periodic) realization of $(\tilde G, \varphi)$ is an
element $(\vec p, \vec L) \in (\mc X^d)^n \times \Hom(\Gamma, \R^d)$ where $\vec p = (\vec p_1, \dots, \vec p_n)$ and $\vec p_i(\gamma)$
is the position of vertex $(i, \gamma) \in V(\tilde G)$. We say that $(\vec p, \vec L) \in (\mc X^d)^n \times \Hom(\Gamma, \R^d)$
is a $\Lambda$-periodic realization if $\lambda^{-1} \vec p - \vec p = \vec L(\lambda)$ for all $\lambda \in \Lambda$ where we view $\vec L(\lambda)$
as a constant function in $\Func(\Gamma, \R^d) \cong \mc X^d$.

We obtain an alternative parameterization as follows. Let $\mc P = \mc X^{dn} \times \Hom(\Gamma,\R^d)$ and define a $\Gamma$-action as follows:
$\gamma \cdot (\vec q, \vec L) = (\gamma \cdot \vec q, \vec L)$. For any subgroup $\Lambda < \Gamma$, let $\mc P_{\Lambda}$ be the subspace of
$\Lambda$-invariant vectors. We define an $\R$-linear isomorphism $\Psi: \mc P \to \mc P$ as
$\Psi(\vec q, \vec L) = ( (\vec q_1 + \vec L, \vec q_2 + \vec L, \dots, \vec q_n + \vec L), \vec L)$. (Note that we can view
$\vec L \in \Hom(\Gamma, \R^d) \subset \Func(\Gamma, \R^d) \cong \mc X^d$.) It is straightforward to check that $\Psi(\mc P_\Lambda)$ is precisely the
space of $\Lambda$-periodic frameworks. We therefore call $(\vec q, \vec L) \in \mc P$ an \emph{alternative parameterization} of the
realization $\Psi(\vec q, \vec L)$. In the following, we will work almost exclusively with $(\vec q, \vec L)$.

\subsection{Length functions and differentials}
For any realization $(\vec p, \vec L)$ of $(\tilde G, \varphi)$ (not requiring any symmetry or periodicity), all lengths (squared) of edges corresponding
to $ij \in E(G)$ can be encoded in the function $\| \vec p_j \cdot \gamma_{ij}^{-1} - \vec p_i \|^2 \in \mc X$ where the value at $\gamma$ is
the squared-length of the edge going from $(i, \gamma)$ to $(j, \gamma \gamma_{ij})$. We therefore define a function $\ell_{ij}: \mc P \to \mc X$
where for $(\vec q, \vec L) \in \mc P$ with $(\vec p, \vec L) = \Psi(\vec q, \vec L)$, we set
$$\ell_{ij}(\vec q, \vec L) = \| \vec p_j \cdot \gamma_{ij}^{-1} - \vec p_i \|^2 = \| (\vec q_j + \vec L) \cdot \gamma_{ij}^{-1} - (\vec q_i + \vec L) \|^2
= \| \vec q_j \cdot \gamma_{ij}^{-1} - \vec q_i + \vec L(\gamma_{ij}) \|^2$$
(Here again, we view $\vec L(\gamma_{ij})$ as a constant function.) It is clear from definitions that $\ell_{ij}$ is $\Gamma$-equivariant
and thus $\ell_{ij}(\mc P_\Lambda) \subseteq \mc X_\Lambda$. Moreover, note that $\mc P_\Lambda, \mc X_\Lambda$ are
preserved by $\Gamma$ (since all $\Lambda < \Gamma$ are normal), so $\ell_{ij}$ is $\Gamma$-equivariant as a map $\mc P_\Lambda \to \mc X_\Lambda$.
We let $\ell: \mc P \to \mc X^m$ be the $m$-tuple of all length functions and set $\ell_\Lambda := \ell|_{\mc P_\Lambda} : \mc P_\Lambda \to \mc X_\Lambda^m$.

For any (alternatively parameterized) $\Lambda$-periodic configuration $(\vec q, \vec L)$, the $\Lambda$-periodic realization space is $\ell_\Lambda^{-1}(\ell_\Lambda(\vec q, \vec L))$
and the space of infinitesimal motions is the kernel of the differential $d \ell_\Lambda$. Thus, the problem of infinitesimal ultrarigidity is determining
when $(\vec q, \vec L) \in \mc P_\Gamma$ induces the minimal possible kernel of $d\ell_\Lambda$ at the point $(\vec q, \vec L)$ over all sublattices $\Lambda < \Gamma$.
Since $\mc P_\Lambda$ and $\mc X_\Lambda$ are finite dimensional linear spaces,
the tangent space at each point for both respectively is naturally isomorphic to $\mc P_\Lambda, \mc X_\Lambda$. Moreover,
the $ij$ coordinate of the differential $d\ell_{\Lambda}(\vec q, \vec L): \mc P_\Lambda \to \mc X_\Lambda^m$ applied to $(\vec v, \vec M)$ is computed to be
$$ \langle \vec v_j \cdot \gamma_{ij}^{-1} - \vec v_i + \vec M(\gamma_{ij}),  \vec q_j \cdot \gamma_{ij}^{-1} - \vec q_i + \vec L(\gamma_{ij}) \rangle.$$

\paragraph{Passing to group rings over finite groups:}
The above computation of $d\ell_{\Lambda}$ applies to any $\Lambda$-periodic realization.
If we know additionally that $(\vec q, \vec L) \in \mc P_\Gamma \subset \mc P_\Lambda$, we can say more. Since
$\ell_\Lambda$ is $\Gamma$-equivariant, the map on tangent bundles
$d\ell_\Lambda: T \mc P_\Lambda \to T \mc X_\Lambda^m$ is too, so for any $\gamma\in \Gamma$
$$  \gamma \cdot d\ell_\Lambda(\vec q, \vec L, \vec v, \vec M) =
d\ell_\Lambda( \gamma \cdot (\vec q, \vec L, \vec v, \vec M)) = d\ell_\Lambda( \gamma \cdot \vec q, \vec L, \gamma \cdot  \vec v, \vec M)) .$$
When $(\vec q, \vec L) \in \mc P_\Gamma$, we also have $ (\gamma \cdot \vec q, \vec L)  = ( \vec q, \vec L)$ and so
$d\ell_{\Lambda}(\vec q, \vec L): \mc P_\Lambda \to \mc X_\Lambda^m$ is $\Gamma$-equivariant.
This can also be verified via the formula for $d\ell_\Lambda$. Specifically, one must use the fact that
$ \vec q_j \cdot \gamma_{ij}^{-1} - \vec q_i + \vec L(\gamma_{ij})  =  \vec q_j - \vec q_i + \vec L(\gamma_{ij})$
is a constant function.

Note that the formula for $d\ell_\Lambda$ makes no reference to $\Lambda$. Thus, for $(\vec q, \vec L) \in \mc P_\Gamma$, we
define $R_{\vec q, \vec L}: \mc P \to \mc X^m$ as
$$ R_{\vec q, \vec L}(\vec v, \vec M) =  \langle \vec v_j \cdot \gamma_{ij}^{-1} - \vec v_i + \vec M(\gamma_{ij}),  \vec q_j - \vec q_i + \vec L(\gamma_{ij}) \rangle.$$
By definition, the map $R_{\vec q, \vec L}$ restricted to $\mc P_\Lambda$ is $d\ell_{\Lambda}(\vec q, \vec L)$,
and so we obtain directly:
\begin{lemma}\lemlab{Rgammasys}
The framework $G(\vec p, \vec L)$ is infinitesimally ultrarigid if and only if the dimension of $\displaystyle \ker(R_{\vec q, \vec L}) \cap \left(\bigcup_{\Lambda <_{f.i.} \Gamma} \mc P_\Lambda \right)$ is $d+1\choose2$ for $(\vec q, \vec L) = \Psi(\vec p, \vec L)$. \qed
\end{lemma}

Since $d\ell_{\Lambda}(\vec q, \vec L)$ is $\Gamma$-equivariant
and $\Lambda$ acts trivially on $\mc P_\Lambda, \mc X_\Lambda$, the map $R_{\vec q, \vec L}$ restricted to $\mc P_\Lambda$ is a map of $\R[\Gamma/\Lambda]$-modules.
We describe the map as follows. It is straightforward to
check that $\mc X_\Lambda \to \R[\Gamma/\Lambda]$ defined by
$$f \mapsto \frac{1}{[\Gamma:\Lambda]} \sum_{[\gamma] \in \Gamma/\Lambda} f(\gamma) [\gamma]$$
is an isomorphism of $\R[\Gamma/\Lambda]$-modules. This moreover induces an isomorphism $\mc P_\Lambda \cong \R[\Gamma/\Lambda] \times \Hom(\Gamma, \R^d)$
where $\Hom(\Gamma, \R^d)$ is taken to be $d^2$ copies of the trivial $\R[\Gamma/\Lambda]$-module. We define a pairing
$[-,-]: \R[\Gamma/\Lambda]^d \times \R[\Gamma/\Lambda]^d \to \R[\Gamma/\Lambda]$ as
$$[ (b_1, \dots, b_d), (c_1, \dots c_d)] = \sum_{k=1}^d b_k c_k.$$
We identify $\R^d \otimes \R[\Gamma/\Lambda] \cong \R[\Gamma/\Lambda]^d$ via $\vec (b_1, \dots, b_d) \otimes c \mapsto (b_1 c, \dots, b_d c)$.
For $(\vec q, \vec L) \in \mc P_\Gamma$, let $\vec d_{ij} = \vec q_j(0) - \vec q_i(0) + \vec L(\gamma_{ij}) \in \R^d$, and let
$\hat R_{\vec q, \vec L} : \R[\Gamma/\Lambda]^{dn} \times \Hom(\Gamma, \R^d) \to  \R[\Gamma/\Lambda]^m$ be the ($\R[\Gamma/\Lambda]$-linear)
map whose $ij$ coordinate is
$$[\vec w_j, \vec d_{ij} \otimes [\gamma_{ij}^{-1}]] - [\vec w_i, \vec d_{ij} \otimes 1] +
\langle \vec M(\gamma_{ij}), \vec d_{ij}\rangle \frac{1}{[\Gamma:\Lambda]}\sum_{[\gamma] \in \Gamma/\Lambda} [\gamma]$$
We remark that $\vec d_{ij}$ is also equal to $\vec p_j(0) - \vec p_i(0) + \vec L(\gamma_{ij})$ for $(\vec p, \vec L) = \psi(\vec q, \vec L)$.

\begin{lemma} \lemlab{tofingrpring}
Let $(\vec q, \vec L) \in \mc P_\Gamma$ and let $R_{\vec q, \vec L}, \hat R_{\vec q, \vec L}$ be defined as above. Then the following diagram commutes:
$$\begin{CD}
\mc P_\Lambda @>R_{\vec q, \vec L}>> \mc X_\Lambda^m \\
@VV\cong V                                               			@VV\cong V  		\\
\R[\Gamma/\Lambda]^{dn} \times \Hom(\Gamma, \R^d) @>\hat R_{\vec q, \vec L}>> \R[\Gamma/\Lambda]^m
\end{CD}$$
\end{lemma}
\begin{proof}
This follows in a straightforward manner from the definitions.
\end{proof}

\paragraph{A few facts from finite representation theory:}
Let $\Delta = \Gamma/\Lambda$ which is a finite abelian group. The ring $\R[\Delta]$ can be identified, as an $\R$-algebra, with a finite direct product
$\R[\Delta] \cong \prod_{k=1}^t A_k$ where each $A_k$ is either $\R$ or $\C$. The corresponding projection $\R[\Delta] \to A_k$ must map
each $\Delta$ to some subgroup of $\C^\times$ generated by a root of unity. Moreover, all such homomorphisms $\Delta \to \C^\times$ (up to complex conjugation)
correspond to some $A_k$.
Since any homomorphism $\Delta \to \C^\times$ is induced by some map $\Gamma \to \C^\times$, for each $k$ there
is a $d$-tuple $\bomega_k = (\zeta_{k,1}, \dots, \zeta_{k,d})$ of roots of unity such that the projection $\R[\Delta] \to A_k$ maps
$[\gamma]$ to $\bomega_k^\gamma := \zeta_{k,1}^{\gamma_1} \cdots \zeta_{k,d}^{\gamma_d}$ where $\gamma_i$ is the
$i$th component of $\gamma$. For convenience, we assume the projection $\R[\Delta] \to A_1 = \R$ is the trivial one sending all $\delta \in \Delta$
to $1$. For any $N$, we can use the above to identify the modules $\R[\Delta]^N \cong \oplus_{k=1}^\ell A_k^N$. The following lemma is
an elementary consequence of the above discussion and representation theory.

\begin{lemma} \lemlab{repthy}
Let $R = \hat R_{\vec q, \vec L}$ for some $(\vec q, \vec L) \in \mc P_\Gamma$.
The map $\displaystyle \R[\Delta]^{dn} \times \Hom(\Gamma, \R^d) \overset{R}{\to} \R[\Delta]^m$ satisfies
\begin{itemize}
\item[(i)] $R(A_1^{dn} \times \Hom(\Gamma, \R)) \subseteq A_1^m$ and $R(A_k^{dn}) \subseteq A_k^m$ for $k \neq 1$
\item[(ii)] For $k \neq 1$, the map $ A_k^{dn} \to A_k^m$ is $\R[\Delta]$-linear and the $ij$ coordinate of $R(\vec w)$ for $\vec w \in A_k^{dn}$
is $$\bm{\omega_k}^{-\gamma_{ij}} \langle \vec d_{ij}, \vec w_j \rangle - \langle \vec d_{ij}, \vec w_i \rangle.$$
\item[(iii)] The map $ A_1^{dn} \times \Hom(\Gamma, \R^d) \to A_k^m$ is $\R[\Delta]$-linear and the $ij$ coordinate of $R(\vec w, \vec M)$ for
$(\vec w, \vec M) \in A_k^{dn} \times \Hom(\Gamma, \R^d)$
is $$ \langle \vec d_{ij}, \vec w_j \rangle - \langle \vec d_{ij}, \vec w_i \rangle + \langle \vec d_{ij}, \vec M(\gamma_{ij}) \rangle$$
\end{itemize}
\end{lemma}

The above lemma tells us that determining infinitesimal ultrarigidity reduces to analyzing two matrices. One matrix is the real
$m \times (dn + d^2)$ matrix, denoted by $S = S_{G, \vec d}$, which, given a colored graph $(G, \bgamma)$
and edge directions $\vec d_{ij}$, has rows given by
$$\begin{array}{cccccc} & i & & j& & \vec M  \\
( \dots & - \vec d_{ij}  & \dots & \vec d_{ij}  & \dots & \gamma_{ij,1} \vec d_{ij}  \; \dots \; \gamma_{ij,d} \vec d_{ij}  ) \end{array}$$
This is the rigidity matrix for periodic rigidity as in \cite{BS10,MT13}. The new data is the matrix with $\R[\Gamma]$ entries,
denoted by $\hat{S} = \hat{S}_{G, \vec d}$, which, given a colored graph $(G, \bgamma)$
and edge directions $\vec d_{ij}$, has rows of the form:
$$\begin{array}{ccccc} & i & & j&  \\
( \dots & - \vec d_{ij}  & \dots & \vec d_{ij} \otimes \gamma_{ij}^{-1} & \dots ) \end{array}$$
For any $\bomega \in \C^d$ which is a $d$-tuple of roots of unity,
there is a unique surjective homomorphism $pr_{\bomega}: \R[\Gamma] \to \F_{\bomega}$ satisfying $pr_{\bomega}(\gamma) = \bomega^\gamma$
where $\F_{\bomega} = \R$ if $\bomega \in \R^d$ and $\F_{\bomega} = \C$ otherwise. For
a matrix with entries in $\R[\Gamma]$, we can apply $pr_{\bomega}$ to each entry. We set $\vec 1 = (1, \dots, 1) \in \C^d$.

An an immediate corollary of \lemref{repthy}, we obtain:
\begin{cor} \corlab{rigiditymatrix}
Let $G(\vec p, \vec L)$ be a periodic framework and $\vec d_{ij}$ the edge vectors. It is infinitesimally
ultrarigid if and only if $S_{G, \vec d}$ has rank $dn + {d\choose2}$ and
$pr_{\bomega}(\hat{S}_{G, \vec d})$ has $\C$-rank $dn$ for all $\bomega \neq \vec 1$.
\end{cor}
Since $S_{G,\vec d}$ having full rank verifies that $G(\vec p, \vec L)$ is infinitesimally rigid
as a periodic framework, we have proved:
\ultramatrixproj*

\paragraph{Substituting polynomials for colors in $S, \hat{S}$}
The ring $\R[\Gamma]$ is easily reinterpreted as a polynomial ring. There is a canonical isomophism $\R[\Gamma] \to \R[x_1^{\pm 1}, \dots, x_d^{\pm 1}]$
which maps $\gamma$ to $\vec x^\gamma := x_1^{\gamma_1} \cdots x_d^{\gamma_d}$. From this viewpoint, $pr_{\bomega}$ is equivalent to evaluating
the polynomial at the point $\bomega$. The matrix $S$ is unchanged and $\hat{S}$ becomes
$$\begin{array}{ccccc} & i & & j&  \\
( \dots & - \vec d_{ij}  & \dots & \vec d_{ij} \otimes \vec x^{-\gamma_{ij}} & \dots ) \end{array}$$

\subsection{Fixed-Lattice and Fixed-Volume Ultrarigidity}
It is easy to specialize the above discussion to get an algebraic criterion for a framework $G(\vec p, \vec L)$ to be {\it infinitesimally fixed-lattice ultrarigid}, i.e.
any $\Lambda$-respecting infinitesimal motions with $\vec M = 0$ are trivial. In this case, we can simply drop the columns for $\vec M$ from $S$
to obtain the right condition. In fact, we can simplify more since $pr_{\vec 1}(\hat{S})$ is precisely that matrix.
Note that since $\vec L$ is fixed, this forbids all trivial motions aside from translations. As alluded to above, the following statement, in slightly
different language, was proven previously by Power \cite{PowerRUM}.

\begin{cor} \corlab{fixlatmatrix}
Let $G(\vec p, \vec L)$ be a periodic framework with edge vectors $\vec d_{ij}$. It is infinitesimally f.l. ultrarigid if and only if
$pr_{\bomega}(\hat{S}_{G, \vec d})$ has $\C$-rank $dn$ for all $\bomega \neq \vec 1$ and $dn - d$ for $\bomega = \vec 1$.
\end{cor}

A framework $G(\vec p, \vec L)$ is {\it infinitesimally fixed-volume ultrarigid} if any $\Lambda$-respecting infinitesimal motions where $\vec M$ does not
(infinitesimally) change the (co)volume of $\vec L(\Gamma)$ are trivial motions. Here, the volume of $\vec L(\Gamma) < \R^d$ is the volume of
$\R^d/\vec L(\Gamma)$ or equivalently $\det( \vec L(e_1) \dots \vec L(e_d))$ where $e_i$ are the standard basis vectors of $\Gamma = \Z^d$. For
f.v. ultrarigidity, we will require that $\vec L$ be full rank, or equivalently that $(\vec L(e_1) \dots \vec L(e_d))$ be invertible.

Of course, any $\vec L \in \Hom(\Gamma, \R^d)$ can be viewed as the matrix $\vec L = (\vec L(e_1) \dots \vec L(e_d)) \in \Mat_d(\R)$ and infinitesimal
motions $\vec M$ of $\vec L$ also lie in $\Mat_d(\R)$. Note that if $\vec L = \Id$, then the infinitesimal motions preserving volume are precisely
the vectors in the tangent space $T_{\Id}(\SL_d(\R))$ which is the lie algebra $\ssl_d(\R)$ of trace $0$ matrices. Thus,
for arbitrary invertible matrices $\vec L$, the infinitesimal motions $\vec M$ preserving volume are those satisfying $\tr(\vec L^{-1} \vec M) = 0$.

\begin{cor} \corlab{fixvolmatrix}
Let $G(\vec p, \vec L)$ be a periodic framework with edge vectors $\vec d_{ij}$. It is infinitesimally f.v. ultrarigid if and only if
the system defined by $S_{G, \vec d}$ and $\tr(\vec L^{-1}\vec M) = 0$ has rank $dn + {d\choose2}$ and  $pr_{\bomega}(\hat{S}_{G, \vec d})$ has $\C$-rank
$dn$ for all $\bomega \neq \vec 1$.
\end{cor}

\begin{remark}
One could alternatively view f.v. ultrarigidity as follows. For each $\Lambda$, we could allow those motions which preserve
the volume of $\vec L(\Lambda)$, not $\vec L(\Gamma)$. However, note that the volume of $\vec L(\Lambda)$ is always a constant multiple of $\vec L(\Gamma)$
as $\vec L$ varies over all possibilities (the multiple is the index), so the two notions are equivalent.
\end{remark}

\paragraph{Affine invariance}
In the cases of a fully flexible lattice or fixed lattice, the dimension of $\Lambda$-respecting motions remains under an affine transformation \cite{BS10}.
Particularly, if $A \in \GL_d(\R)$, then $G(\vec p, \vec L)$ and $G(A \cdot \vec p, A \circ \vec L)$ have the same dimension of $\Lambda$-respecting motions
where $(A \cdot \vec p_i)(\gamma) = A(\vec p_i(\gamma))$. The dimension of
motions is {\it not} preserved by affine transformations in the case of the fixed-volume
lattice. In fact, this failure is an integral part in establishing a Maxwell-Laman type theorem for fixed-area rigidity in dimension $2$ \cite{MT14}.

\subsection{Connection to the RUM spectrum}
Viewing $\hat S$ as a matrix with polynomial entries, we can consider the rank after evaluating $\vec x$ at any vector $\omega \in (\C^\times)^d$.
In \cite{PowerRUM}, Power defines the RUM (Rigid Unit Mode) spectrum of a framework $G(\vec p, \vec L)$ to be the subset of vectors $\vec k = (k_1, \dots, k_d) \in [0,1)^d$
such that the matrix $\hat S$ evaluated at $ \vec x = (\exp(2\pi i k_1), \dots, \exp(2\pi i k_d))$ has nontrivial kernel. Those points in the RUM spectrum with rational coordinates
(the rational RUM spectrum) correspond precisely to torsion points. The algorithm described in \secref{algo} thus determines when the rational RUM spectrum of
a framework is trivial.

The term rigid unit mode is also used to describe certain kinds of low-energy phonons of certain crystalline materials, which have been studied by
Dove et al \cite{Dove1}, Giddy et al \cite{Giddy1}, Hammonds et al \cite{Hammonds1, Hammonds2}, and Swainson and Dove \cite{SwainDove}. For the
precise connection between these two notions, we refer the reader to \cite[Section 6]{PowerRUM}.

\section{Algorithmic detection of infinitesimal rigidity}\seclab{algo}

In this section, we establish our algorithm for checking infinitesimal ultrarigidity in time polynomial in the degrees of the
minors. The key fact (\lemref{deeptoshallow}) to be proved is that if a polynomial has no torsion points up to a certain order
except $\vec 1$, then it has no
torsion points at all except $\vec 1$. The proof of this fact uses a few ideas from the proof of a theorem
of Liardet \cite{Lang, Liardet} which shows that if the variety of a polynomial of two variables has a torsion point of high order, then it contains an entire torsion coset.
As a consequence of our work below, we prove an analogue of this theorem for arbitrarily many variables with explicit estimates.

\subsection{Preliminary facts about lattices}

For a lattice $\Lambda \subseteq \R^d$, the {\it volume} of $\Lambda$, denoted $\vol(\Lambda)$ is the
volume of $\R^d/\Lambda$. This is also known as the {\it determinant} of $\Lambda$ since it is
the determinant of any $d \times d$ matrix whose columns are a basis of $\Lambda$.
If $\Lambda \subset \R^d$ is discrete but not a lattice, we set $\vol(\Lambda) = \vol(\R \cdot \Lambda/\Lambda)$.
The following theorem
of \cite{LLL82} implies that there is a basis of $\Lambda$ which is as ``small'' as its volume.
Let $\| \cdot \|_2$ denote the standard $L^2$-norm (i.e. Euclidean norm) on $\R^d$.
\begin{theorem}[name={\cite{LLL82}}] \label{theorem:smallbasis}
Let $\Lambda \subseteq \R^d$ be a lattice. There exists a basis $\lambda_1, \dots, \lambda_d$ of $\Lambda$
such that
$$\prod_{i=1}^d \| \lambda_i \|_2 \leq {\small \left(\frac{4}{3}\right)}^{d(d-1)/4} \vol(\Lambda)$$
\end{theorem}

\begin{lemma} \lemlab{volgeq1}
Suppose $\{0\} \neq \Lambda$ is a subgroup of $\Z^d \subset \R^d$. Then, $\vol(\Lambda) \geq 1$.
\end{lemma}
\begin{proof}
If $\Lambda$ has rank $d$, then $\vol(\Lambda) = [\Z^d:\Lambda] \vol(\Z^d) = [\Z^d: \Lambda] \geq 1$. If
$\rk(\Lambda) = k < d$, then there is a subset $e_{i_1}, \dots, e_{i_{d-k}}$ of standard basis vectors such that
$\Lambda$ and $e_{i_1}, \dots, e_{i_{d-k}}$ generate a rank $d$ subgroup $\Lambda'$. We have
$$ 1 \leq \vol(\Lambda') \leq \vol(\Lambda) \prod_{\ell = 1}^{d-k} \|e_{i_\ell} \|_2 = \vol(\Lambda).$$
\end{proof}

\subsection{Some preliminaries on torsion points and torsion cosets}
We henceforth set $U = (\C^\times)^d \subset \C^d$.
For any point $\vec a = (a_1, \dots, a_d) \in U$ and integer point
$\lambda = (\lambda_1, \dots, \lambda_d) \in \Z^d$, we set
$$\vec a^\lambda := \prod_{i=1}^d a_i^{\lambda_i}.$$
Recall that $\bomega \in U$ is a torsion point if $\bomega = (\zeta_1, \dots, \zeta_d)$ where all $\zeta_i$ are roots of unity, i.e.
$\bomega$ is a finite order element in the multiplicative group $U$.
A {\it torsion coset} is a subvariety of $U$ of the form $V_U(x^{\lambda_i} - \eta_i)$ where the $\lambda_i$ generate
a direct summand of $\Z^d$ and $\eta_i$ are roots of unity.

\begin{lemma} \label{lemma:tomthroot}
Let $\Lambda' < \Lambda$ be subgroups of rank $k$ in $\Z^d$ and let $M = [\Lambda: \Lambda']$. If $\bomega^{\lambda'} = 1$
for all $\lambda' \in \Lambda'$, then $\bomega^{\lambda}$ is an $M$th root of unity for all $\lambda \in \Lambda$.
\end{lemma}
\begin{proof}
For any $\lambda \in \Lambda$, we have $M \lambda \in \Lambda'$. Thus,
$(\bomega^\lambda)^M = \bomega^{M \lambda} = 1$.
\end{proof}

The ring of regular functions $\C(U)$ is  $\C[x_1^{\pm 1}, \dots, x_d^{\pm 1}]$. For any collection $q_1, \dots, q_k \in \C[x_1^{\pm 1}, \dots, x_d^{\pm 1}]$, we
denote the zero set in $U$ by $V_U(q_1, \dots, q_k)$.

\begin{lemma} \lemlab{irredtorcoset}
Let $\Lambda \subseteq \Z^d$ be a rank $k$ subgroup with generators  $\lambda_1, \dots \lambda_k$ and let $\eta_1, \dots, \eta_k$ be roots of unity.
If $\Lambda$ is a direct summand of $\Z^d$, then $V_U(x^{\lambda_i} - \eta_i : i\in [k])$ is an irreducible quasi-projective variety.
\end{lemma}
\begin{proof}
There exists a (non-unique) automorphism $\Z^d \to \Z^d$ mapping $\lambda_i \mapsto e_i$ where $e_i$ is the standard generator.
This induces an automorphism $\varphi$ of $\C[x_1^{\pm 1}, \dots, x_d^{\pm 1}]$ satisfying $\varphi(x^{\lambda_i}) = x_i$.
Thus, under $\varphi$, the ideal $(x^{\lambda_1} - \eta_1, \dots, x^{\lambda_k} - \eta_k)$ is the preimage of
$(x_1 - \eta_1, \dots, x_k - \eta_k)$ which is prime.
\end{proof}

\subsection{Bezout's inequality in affine space}
We recall the notion of degree from \cite{Heintz}. One particular advantage we will use is that degree is defined for any variety
without requiring knowledge of the defining polynomials. Note that Heintz defines degree for any ``constructible'' set, but
varieties will suffice for us.

\begin{definition}
Let $X \subset \C^d$ be an irreducible variety of dimension $r$. Then
$$\deg(X) = \sup\{  |E \cap X|  \;\; : \;\; E \text{ is a } (d-r) \text{-dimensional affine subspace such that } E \cap X \text{ is finite} \}$$
For $X$ reducible with components $X_1, \dots, X_c$,
$$\deg(X) = \sum_{i=1}^c \deg(X_i)$$
\end{definition}

We state some basic facts about degree.
\begin{itemize}
\item If $X = V(p)$, then $\deg(X) = \deg(p)$ \cite[Remark 2.(3)]{Heintz}.
\item If $X$ is finite then $\deg(X) = |X|$.
\end{itemize}

We can phrase Bezout's inequality as follows.

\begin{theorem}[name={\cite[Theorem 1]{Heintz}}]  \label{theorem:bezout}
Let $X, Y$ be subvarieties of $\C^d$. Then, $\deg(X \cap Y) \leq \deg(X) \cdot \deg(Y)$.
\end{theorem}

We will apply this theorem to our particular situation of varieties in $U$. We define a kind of degree for polynomials in
$\C[x_1^{\pm 1}, \dots, x_d^{\pm 1}]$. We set
$$\tdeg(p) = \min_{\gamma \in \Z^d} (\deg( x^{\gamma} p) \;\; | \;\; x^{\gamma} p \in \C[x_1, \dots, x_d])$$
where $\deg$ on the right hand side is the usual degree of a polynomial.

For any $\lambda = (\ell_1, \dots, \ell_d) \in \Z^d$, let $\ell_i^+ = \ell_i$ if $\ell_i > 0$ and let $\ell_i^+ = 0$ otherwise. Let $\ell_i^- = -\ell_i$
if $\ell_i < 0$ and let $\ell_i^- = 0$ otherwise. Set $\lambda^+ = (\ell_1^+, \dots, \ell_d^+)$ and $\lambda^- = (\ell_1^-, \dots, \ell_d^-)$.
It follows that $\lambda^+$ and $\lambda^-$ have disjoint support and are nonnegative vectors, and that $\lambda = \lambda^+ - \lambda^-$.

\begin{lemma} \lemlab{BezoutinU}
Let $\lambda_1, \dots, \lambda_{d-1}$ generate a summand of $\Z^d$, and let $\eta_1, \dots, \eta_{d-1} \in \C$ be roots of unity.
Let $p \in \C[x_1^{\pm 1}, \dots, x_d^{\pm 1}]$, and set $q_i = x^{\lambda_i} - \eta_i$ for $1 \leq i \leq d-1$.
Then, either $V_U(p) \supset V_U(q_1, \dots, q_{d-1})$ or
$$| V_U(p) \cap V_U(q_1, \dots, q_{d-1})| \leq \tdeg(p) \cdot \prod_{i=1}^{d-1} \| \lambda_i \|_1.$$
\end{lemma}
\begin{proof}
Let $Y = V_U(q_1, \dots, q_{d-1})$.
By \lemref{irredtorcoset}, $Y$ is a $1$-dimensional irreducible quasi-projective variety.
Consequently, $V_U(p) \cap Y$ is either $Y$ or a finite set of points. It suffices to show that if the intersection is finite,
then $|V_U(p) \cap Y| \leq \tdeg(p) \cdot \prod_{i=1}^{d-1} \| \lambda_i \|_1$. So w.l.o.g. assume the intersection is finite.

We bound degrees. Let $\tilde{q}_i = x^{\lambda_i^+} - \eta_i x^{\lambda_i^-} \in \C[x_1, \dots, x_d]$ and let
$X = V(\tilde{q}_1, \dots, \tilde{q}_{d-1})$. Let $\overline{Y}$
be the Zariski closure of $Y$ in $\C^d$. Clearly, $\overline{Y}$ is an irreducible component of  $X$,
so $\deg(\overline{Y}) \leq \deg(X)$, and by Bezout's inequality
$$\deg(X) \leq \prod_{i=1}^{d-1} \deg(\tilde{q}_i) \leq \prod_{i=1}^{d-1} \|\lambda_i \|_1.$$

Let $\tilde{p} = x^{\gamma} p$ such that $\deg(\tilde p) = \tdeg(p)$ and $p \in \C[x_1, \dots, x_d]$.
Then, $V_U(p) = V_U(\tilde{p})$. By Bezout's inequality
$$\deg(V(\tilde{p}) \cap \overline{Y}) \leq  \deg(V(\tilde{p})) \deg(\overline{Y}) \leq \tdeg(p)  \prod_{i=1}^{d-1} \|\lambda_i \|.$$
The lemma now follows from the ``basic facts''.
\end{proof}

\subsection{Torsion points in varieties} \seclab{tpiv}

The key algebraic lemma required for our algorithm is the following. To condense notation, we set $C_d = \left(\frac{4}{3}\right)^{(d-1) (2d-3)/4} d^{(d-1)/2}$.
\begin{lemma} \lemlab{deeptoshallow}
Let $p \in \Q[x_1^{\pm1}, \dots, x_d^{\pm1}]$. Suppose $V(p)$ contains a torsion point $\bomega$ of order $N$ with
$$\phi(N) > C_d \tdeg(p) N^{(d-1)/d} .$$
Then $V(p)$ contains a torsion point $\bomega' \neq \vec 1$
of order $M < N$ where $\bomega', M$ depend only on $\bomega$.
\end{lemma}

To prove this, we show that any torsion point of sufficiently high order is contained in a one-dimensional torsion coset defined by polynomials of relatively
small degree. Moreover, we ensure that the torsion coset contains a torsion point of lower order. The small degrees of the polynomials then allows
us to use Bezout's inequality. We denote the $\ell_1$ norm of a vector $\gamma \in \Z^d$ by $\| \gamma \|_1$.

\begin{lemma} \label{lemma:torsioncoset}
Let $\bomega$ be a torsion point of order $N$ where $N^{1/d} > \left(\frac{4}{3}\right)^{\frac{(d-1)^2}{4}}$.
For some $M< N$, there exist $M$th roots of unity $\eta_1, \dots, \eta_{d-1}$ and vectors $ \lambda_1, \dots, \lambda_{d-1} \in \Z^d$ such that
\begin{itemize}
\item $\bomega$ is a zero of $x^{\lambda_i} - \eta_i$ for all $i = 1, \dots d-1$,
\item $\prod_{i=1}^{d-1} \| \lambda_i \|_1 \leq \frac{C_d}{M} N^{(d-1)/d} $
\item $\lambda_1, \dots, \lambda_{d-1}$ generate a summand of $\Z^d$
\end{itemize}

\end{lemma}

\begin{proof}
By assumption, there is a primitive $N$th root of unity $\zeta$ and $\kappa = (k_1, \dots, k_d) \in \Z^d$ such that $\bomega = (\zeta^{k_1}, \dots, \zeta^{k_d})$.
Let $\Gamma' = \{ \gamma \in \Z^d \; | \; \gamma \cdot \kappa \equiv 0 \text{ mod } N \}$ which is precisely the set of integer vectors
satisfying $\bomega^\gamma = \vec 1 $. Note that $\gcd(k_1, \dots, k_d, N) = 1$, and so there is some $\gamma \in \Z^d$ such that
$\gamma \cdot \kappa = 1$ (mod $N$). Thus, $\Gamma'$ has index $N$ in $\Z^d$, and $\vol(\Gamma') = N$.

By Theorem \ref{theorem:smallbasis}, there is a basis $\gamma_1', \dots, \gamma_d'$ of $\Gamma'$ such that
$$\prod_{i=1}^d \| \gamma_i '\|_2 \leq \left(\frac{4}{3}\right)^{d(d-1)/4} N.$$
Without loss of generality, assume $\|\gamma_1 '\|_2 \leq \| \gamma_2' \|_2 \leq \dots \leq \| \gamma_d' \|_2$, and
set $\Lambda' = \langle \gamma_1' , \dots,  \gamma_{d-1}' \rangle$. With this assumption,
$$\prod_{i=1}^{d-1}  \| \gamma_i' \|_2 \leq ( \left(\frac{4}{3}\right)^{d(d-1)/4} N)^{(d-1)/d} = \left(\frac{4}{3}\right)^{(d-1)^2/4} N^{(d-1)/d}.$$
Let $\Lambda = \{ \lambda \in \Z^d \; | \; s \lambda \in \Lambda' \text{ for some } 0 \neq s \in \Z \}$. We now establish
some claims about $\Lambda$.

\paragraph{Claim 1:} $M := [\Lambda: \Lambda'] \leq \left(\frac{4}{3}\right)^{\frac{(d-1)^2}{4}} N^{(d-1)/d} < N$ \\

From \lemref{volgeq1}, we obtain $M = [\Lambda: \Lambda'] = \vol(\Lambda')/\vol(\Lambda) \leq \vol(\Lambda')$.
By Hadamard's inequality,
$$\vol(\Lambda') \leq \prod_{i=1}^{d-1}  \| \gamma_i' \|_2 \leq \left(\frac{4}{3}\right)^{(d-1)^2/4} N^{(d-1)/d} < N.$$

\paragraph{Claim 2:} There is a basis $\lambda_1, \dots, \lambda_{d-1}$ of $\Lambda$ satisfying
$$\prod_{i=1}^{d-1} \| \lambda_i \|_1 \leq \frac{C_d}{M} N^{(d-1)/d}.$$

\noindent By Theorem \ref{theorem:smallbasis},
$\Lambda$ has a basis $\lambda_1, \dots, \lambda_{d-1}$ satisfying $\prod_{i=1}^{d-1} \| \lambda_i \|_2 \leq \left(\frac{4}{3}\right)^{(d-1)(d-2)/4} \vol(\Lambda)$.
We also have $\vol(\Lambda) = \vol(\Lambda')/M$, and by Hadamard's inequality,
$\vol(\Lambda') \leq \prod_{i=1}^{d-1} \| \gamma_i' \|_2 \leq \left(\frac{4}{3}\right)^{\frac{(d-1)^2}{4}} N^{(d-1)/d}$. These inequalities and
the fact that $\|\vec v \|_1 \leq d^{1/2}\| \vec v \|_2$ establish Claim 2.

We are now essentially finished. By Lemma \ref{lemma:tomthroot}, $\eta_i =\bomega^{\lambda_i}$ is an $M$th root of unity,
and $\bomega$ is a zero of $x^{\lambda_i} - \eta_i$. By definition of $\Lambda$, it is necessarily a direct summand of $\Z^d$.
\end{proof}

\begin{proof}[Proof of \lemref{deeptoshallow}]
As in the previous proof $\bomega = (\zeta^{k_1}, \dots, \zeta^{k_d})$ where $\zeta$ is a primitive $N$th root of unity
and $\gcd(k_1, \dots, k_d, N) = 1$.
The lemma will follow essentially from the combination of \lemref{BezoutinU} and Lemma \ref{lemma:torsioncoset}.

We first set up the polynomials defining a torsion coset. Note that $N > \phi(N)$, and so it follows from the hypothesis
that $N^{1/d} > \left(\frac{4}{3}\right)^{\frac{(d-1)^2}{4}}$. Let $\lambda_i, \eta_i, M$ for $1 \leq i \leq d-1$ be as in Lemma \ref{lemma:torsioncoset}.
Let $q_i = x^{\lambda_i} - \eta_i$, and set $Y = V_U(q_1,\dots, q_{d-1})$. Let $\eta$ be some primitive $M$th root of unity
and write $\eta_i = \eta^{m_i}$ for some $m_i \in \Z$.

We estimate $|V_U(p) \cap Y|$. Since
the coefficients of $p$ and the $q_i$ lie in $\Q(\eta)$, for any $\sigma \in \Galois(\Q(\zeta)/\Q(\eta))$ we have
$p(\sigma(\bomega)) = \sigma(p(\bomega)) = 0$ and $q_i(\sigma(\bomega)) = \sigma(q_i(\bomega)) = 0$. Since
$\langle \zeta \rangle = \langle \zeta^{k_1}, \dots, \zeta^{k_d} \rangle$, any Galois automorphism fixing
$\bomega$ also fixes $\Q(\zeta)$. Consequently, the $\Galois(\Q(\zeta)/\Q(\eta))$ orbit of $\bomega$ has size
$|\Galois(\Q(\zeta)/\Q(\eta))| = \phi(N)/\phi(M)$. It follows that
\begin{align} | V_U(p) \cap Y | & \geq  \phi(N)/\phi(M) > \phi(N)/M \\
&  >  \frac{C_d}{M} N^{(d-1)/d}   \tdeg(p) \\
&\geq  \tdeg(p) \prod_{i=1}^{d-1} \| \lambda_i \|_1 .\end{align}
By \lemref{BezoutinU}, $V_U(p) \supset Y$.

It remains to show that $Y$ contains a torsion point $\neq \vec 1$ whose coordinates are $M$th roots of unity. First suppose $M > 1$.
Since $\Lambda = \langle \lambda_1, \dots, \lambda_{d-1}\rangle$ is a direct summand of $\Z^d$, there is a vector $\lambda_d$ which extends $\lambda_1, \dots, \lambda_{d-1}$
to a basis of $\Z^d$. Let $q_d(x) = x^{\lambda_d} - 1$. If we identify $\langle \eta \rangle \cong \Z/M\Z$, then the system of equations
$q_1 = \dots = q_d = 0$ restricted to $\langle \eta \rangle^d \subset \C^d$ is equivalent to the $\Z/M\Z$-linear system
$$ \left( \begin{array}{c} \lambda_1 \\ \vdots \\ \lambda_{d-1} \\ \lambda_d \end{array} \right) X =  \left( \begin{array}{c} m_1 \\ \vdots \\ m_{d-1} \\ 0  \end{array} \right).$$
Since the matrix is invertible in $\Z$, it is invertible as a matrix in $\Z/M\Z$, and so there is some solution.

Suppose instead $M = 1$. Then $q_i = x^{\lambda_i} -1$ for all $i \leq d-1$. Set instead $q_d = x^{\lambda_d} +1$. Then the above argument shows that
$Y$ contains some torsion point in $\{ \pm 1 \}^d$ which is not $\vec 1$.
\end{proof}

Although $p$ was assumed to be a rational polynomial for \lemref{deeptoshallow}, the lemma can be modified for any complex polynomial. The algorithm
will then extend if the field generated by the coefficients of $p$ can be sufficiently understood. We let $\Q_{ab}$ be the field generated over $\Q$ by all
roots of unity.
\begin{lemma} \lemlab{deeptoshallowgeneral}
Let $p \in \C[x_1^{\pm1}, \dots, x_d^{\pm1}]$, and let $K\subset \C$ be the field generated by $\Q$ and the coefficients of $p$.
Suppose $V(p)$ contains a torsion point $\bomega$ of order $N$ satisfying
$$\phi(N) > C_d \tdeg(p) [K \cap \Q_{ab}: \Q] N^{(d-1)/d}  ,$$
Then $V(p)$ contains a torsion point
$\bomega' \neq \vec 1$ of order $M < N$ where $\bomega', M$ depend only on $\bomega$.
\end{lemma}
\begin{proof} Apart from the paragraph beginning with ``We estimate $|V_U(p) \cap Y|$...'', the argument for \lemref{deeptoshallow} applies. We replace the
aforementioned paragraph with the following. Let $K' = K \cap \Q_{ab}$.

First, we need to show $[K(\zeta): K] = [K'(\zeta): K']$. Let $f(x)$ be the minimal polynomial of $\zeta$ over $K$, and let $g(x)$ be the minimal polynomial
of $\zeta$ over $\Q$. All the roots of $g$ are powers of $\zeta$, and since $f$ necessarily divides $g$, the same holds for $f$.
Consequently, $f \in \Q(\zeta)[x]$, and so $f \in K'[x]$, and this implies $f$ is a minimal polynomial for $\zeta$ over $K'$. Thus, $[K(\zeta): K] = \deg(f) =  [K'(\zeta): K']$.

Next, we estimate $|V_U(p) \cap Y|$. Since
the coefficients of $p$ and the $q_i$ lie in $K(\eta)$, for any $\sigma \in \Galois(K(\zeta)/K(\eta))$ we have
$p(\sigma(\bomega)) = \sigma(p(\bomega)) = 0$ and $q_i(\sigma(\bomega)) = \sigma(q_i(\bomega)) = 0$. Since
$\langle \zeta \rangle = \langle \zeta^{m_1}, \dots, \zeta^{m_d} \rangle$, any Galois automorphism fixing
$\bomega$ and $K$ also fixes $K(\zeta)$. Consequently, the $\Galois(K(\zeta)/K(\eta))$ orbit of $\bomega$ has size
$|\Galois(K(\zeta)/K(\eta))|$. Note that adjoining any root of unity (to a characteristic $0$ field) results in a Galois extension,
and so
$$ |\Galois(K(\zeta)/K(\eta))| = \frac{[K(\zeta): K]}{[K(\eta):K]} \geq \frac{[K'(\zeta):K']}{\phi(M)} > \frac{[K'(\zeta): \Q]}{[K':\Q] M} \geq \frac{\phi(N)}{[K':\Q]M}$$
It follows that
$$  | V(p) \cap V(q_1, \dots, q_{d-1}) |   >     \tdeg(p) \prod_{i=1}^{d-1} \| \lambda_i \|_1.$$
By \lemref{BezoutinU}, $V_U(p) \supset Y$.
\end{proof}

\subsection{Effective estimates for excluding torsion points}

\begin{prop} \proplab{simplealg}
Let $p_1, \dots, p_n \in \Q[x_1^{\pm 1}, \dots, x_d^{\pm 1}]$ and set $\hat C = C_d \max(\tdeg(p_i))$.
Let $N_0$ be sufficiently large such that $N > N_0 \Rightarrow \phi(N) >  \hat C N^{(d-1)/d} $. If $\bomega \notin V(p_1, \dots, p_n)$
for all torsion points $\bomega \neq \vec 1$ of order $N \leq N_0$, then $V(p_1, \dots, p_n)$ cannot contain any torsion point except $\vec 1$.
\end{prop}
\begin{proof} This is a straightforward consequence of \lemref{deeptoshallow}.
\end{proof}

Using the more general \lemref{deeptoshallowgeneral}, we obtain the following.
\begin{prop} \proplab{simplealggeneral}
Let $p_1, \dots, p_n \in \C[x_1^{\pm 1}, \dots, x_d^{\pm 1}]$ and set $\hat C = C_d  [K \cap \Q_{ab}: \Q] \max(\tdeg(p_i))$ where
$K \subset \C$ is the field generated by
$\Q$ and the coefficients of $p_1, \dots, p_n$.
Let $N_0$ be sufficiently large such that $N > N_0 \Rightarrow \phi(N) > \hat C N^{(d-1)/d} $. If $\bomega \notin V(p_1, \dots, p_n)$
for all torsion points $\bomega \neq \vec 1$ of order $N \leq N_0$, then $V(p_1, \dots, p_n)$ cannot contain any torsion point except $\vec 1$.
\end{prop}

Note that $K$ is a finitely generated extension of $\Q$, and so by standard results it follows that $K \cap \Q_{ab} \subset K$ is finitely generated over $\Q$.
Thus, $[K \cap \Q_{ab}: \Q]$ is finite.

\paragraph{A few explicit estimates}
To make effective use of \propref{simplealg}, one needs some estimate of a sufficiently large $N_0$. To do this, we can use some elementary computations
and the following lower bound (see e.g. \cite[Section 4.I.C]{Ribenboim}) where $\gamma$ is Euler's constant
\begin{equation} \phi(N) \geq \frac{N}{e^\gamma \log \log N + \frac{3}{\log \log N}}. \eqlab{philowb} \end{equation}

\begin{lemma} \lemlab{estone}
Let $d \geq 2$ and let $g_d(y) = \sqrt[d]{y}/(e^\gamma \log \log y + \frac{3}{\log \log y}) $. Then $g_d( (y \log y)^d) > y$ for
\begin{itemize}
\item $y \geq 256 d^4$ if $d \geq 4$
\item $y \geq 8500 $ if $d = 2, 3$
\end{itemize}
\end{lemma}
\begin{proof}
First, we note that $\sqrt[d]{e^{e^{3/e^\gamma}}} < 8500$ for $d = 2,3$ and $\sqrt[d]{e^{e^{3/e^\gamma}}} < 256 d^4$
for $d \geq 4$. Thus, $\frac{3}{\log \log y} < e^\gamma$ for the specified values of $y$.

We compute:
$$ g_d( (y \log y)^d)  = y \frac{\log y}{e^\gamma \log \log (y \log y)^d + \frac{3}{\log \log (y \log y)^d}}.$$
For our domain of $y$-values, it therefore suffices to show $\log y > e^\gamma (\log \log (y \log y)^d + 1)$ or
equivalently $y^{1/e^\gamma} > e \log(y \log y)^d = d e \log (y \log y)$.
We will show the stronger inequality $y^{1/2} > d e \log (y \log y)$.

Set $h_d(y)= d e \log (y \log y)$. We first show $\frac{1}{2 \sqrt{y}} \geq h_d'(y)$ for $y \geq 81 d^2$ (which
includes our specified domain). We compute
$$h_d'(y) = \frac{de(\log y + 1)}{y\log y} = \frac{1}{2 \sqrt{y}} \left( \frac{2 d e}{\sqrt{y}} \right) \left( \frac{\log y + 1}{\log y} \right)
\leq  \frac{1}{2 \sqrt{y}} \left( \frac{2 d e}{9 d} \right) \left(\frac{3}{2}\right) \leq  \frac{1}{2 \sqrt{y}} .$$

It remains to show that $h_d(256 d^4) < \sqrt{ 256 d^4} = 16 d^2$ for all $d \geq 4$ and
$h_d(8500) < \sqrt{8500}$ for $d = 2,3$. The latter can be checked by direct computation, so we prove the former.
Note that $h_d(y) \leq 2 d e \log y$. By computation of derivatives (in $d$), the quantity $\psi(d) = 2 e \log (4d)^4 = 8 e \log 4d$
is seen to grow more slowly than $16 d$ for $d \geq 2$, and a direct computation shows $\psi(4) < 64$.
Thus $\psi(d) < 16 d$ for all $d \geq 4$ and $h_d(256 d^4) < 16 d^2$ for all $d \geq 4$.
\end{proof}
We are now ready to state and prove a more explicit version of \theoref{torsionpoints} from
the introduction.
\begin{cor} \corlab{esttwo}
Let $p_1, \dots, p_n \in \C[x_1^{\pm 1}, \dots, x_d^{\pm 1}]$ and set $\hat C =  C_d  [K \cap \Q_{ab}: \Q] \max(\tdeg(p_i))$ where
$K \subset \C$ is the field generated by $\Q$ and the coefficients of $p_1, \dots, p_n$.  Let $N_0 = \max(8500, (\hat C \log \hat C)^d)$ if $d =2,3$
and $N_0 = \max(256 d^4, (\hat C \log \hat C)^d)$ if $d \geq 4$.
If $\bomega \notin V(p_1, \dots, p_n)$
for all torsion points $\bomega \neq \vec 1$ of order $N \leq N_0$, then then $V(p_1, \dots, p_n)$ cannot contain any torsion point except $\vec 1$.
\end{cor}
\begin{proof}
It suffices to prove that $N > N_0$ implies $\phi(N) > D N^{(d-1)/d}$. Because of \eqref{philowb}, it is sufficient that
$g_d(N) > D$. For any $N > N_0$, there is a unique $y > D$ such that $N = (y \log y)^d$, and by \lemref{estone},
$g_d(N) > y > D$.
\end{proof}

\subsection{The algorithm}\seclab{thealgorithm}
From \propref{simplealggeneral} and \corref{esttwo}, there is a clear path for designing a ``brute force'' algorithm for checking infinitesimal ultrarigidity
of a framework. Here, we outline the algorithm, check correctness, and compute the running time. For simplicity, we will describe the algorithm for
the fixed lattice and rational configurations. We discuss modifications of the algorithm for more general input at the end of the section.

The input for the algorithm is a colored graph $(G, \bgamma_{ij})$ and framework $G(\vec p, \vec L)$, so for our purposes we will evaluate
the running time in terms of $m$ and $D = \sum_{ij} \| \gamma_{ij} \|_1$. Moreover, we will work under the assumption
of some fixed dimension $d$. However, it should be noted that the constants can be quite large and grow exponentially in $d$.
We will show that the running time is polynomial in $m$ and $D$. Since the input size required for $\bgamma$ is
$\log D$, our algorithm is technically exponential time. \\

\noindent \textbf{Steps in Algorithm:} \\

\noindent \textbf{I}. Compute $D =  \sum_{ij} \| \gamma_{ij} \|_1$ and compute $N_0$ such that $N > N_0 \Rightarrow \phi(N) >  C_d N^{d-1/d} D$.
From \lemref{estone}, letting $\hat C = C_d D$, we can use $N_0 = \max(8500, (\hat C \log \hat C)^d)$ for $d = 2, 3$ and $N_0 = \max(256d^4, (\hat C \log \hat C)^d)$
for $d \geq 4$.\\

\noindent \textbf{II}. For each integer $N$ from $1$ to $N_0$, do the following.
\begin{itemize}
\item[(a)] Check if $\phi(N) > \hat C N^{(d-1)/d}$ and skip the next computations for $N$ if true.
\item[(b)] Compute $\div(N)$, the set of divisors of $N$.
\item[(c)] Compute the minimum polynomial $m_N(x)$ for $\zeta$ the primitive $N$th root of unity.
\item[(d)] For each $d$-tuple $\bomega = (\zeta^{k_1}, \zeta^{k_2}, \dots, \zeta^{k_d})$ with $k_1 \in \div(N)$ and $0 \leq k_i \leq N$, do the following
\begin{itemize}
\item[(i)] Construct the matrix $\pr_{\bomega}(S)$ where elements of $\Q(\zeta)$ are represented as vectors in the $\Q$-coordinate system from the basis
$\{1, \zeta, \zeta^2, \dots, \zeta^{\phi(N)-1}\}$.
\item[(ii)] Compute the rank of the determinant of $\pr_{\bomega}(S)$. Stop running if it is not full rank and otherwise keep running.
\end{itemize}

\end{itemize}

\noindent \textbf{III} If the algorithm ran through step II for $N$ up to $N_0$, then the framework is infinitesimally ultrarigid and otherwise flexible.

\paragraph{Correctness:}

This follows in a straightforward manner from \propref{simplealg} once one verifies that $\widetilde{\deg}$ of any minor is at most $D$.
The only other point which may require additional explanation is the claim that
we only need to check torsion points $\bomega$ where $k_1$ is a divisor of $N$. However, since we assumed the configuration is rational, the
minors are rational polynomials, and so they evaluate to $0$ at any torsion point $\bomega$ if and only if they do so at any Galois conjugate.
Every Galois orbit contains a torsion point satisfying $k_1 \in \div(N)$.

\paragraph{Running Time:}
We evaluate the running time for each step. As we will see, step II.d dominates rather strongly, so we will give somewhat loose estimates for the other steps.\\

\noindent \textbf{I} The value $D$ is computed from adding positive integers and so must take time $O(D)$. The computation of $N_0$ occurs in constant time.\\

\noindent \textbf{II.ab} The value $\phi(N)$ can be computed in time at most $O(N)$ from a prime factorization which itself can be done in $O(N)$ time.
The divisors $\div(N)$ are computable in time $O(N)$.\\

\noindent \textbf{II.c} Using the prime factorization of $N$, and the following facts, $m_N(x)$ can be computed in time at worst $O(N^2 \log N)$
\begin{itemize}
\item $m_k(x) = x^{k-1} + x^{k-2} + \dots X + 1$ if $k$ is prime
\item $m_{qk}(x) = m_k(x^q)/m_k(x)$ if $q$ is a prime not dividing $k$
\item $m_{qk}(x) = m_k(x^q)$ if $q$ is a prime dividing $k$.
\end{itemize}

\noindent \textbf{II.d preprocessing} Since we represent elements of $\Q(\zeta)$ as polynomials in $\zeta^0, \zeta^1, \dots, \zeta^{\phi(N)-1}$, multiplications in general
take time $O(\phi(N)^2)$ (with $O(\phi(N)^2)$ arithmetic operations and $O(\phi(N)^2)$ for reduction using $m_N(x)$). Before computing the ranks over
various order $N$ torsion, we compute beforehand the following.
\begin{itemize}
\item $\pr_{\zeta^{k_1}, 1, \dots, 1}(\hat S)$ for all $k_1 \in \div(N)$. Fix some $k_1$. For each row in $\hat S$, we must compute at most one algebraic number of the form
$\zeta^{k_1 \ell}$ where $\ell \leq D$. Using $\zeta^N = 1$, we can assume $0 \leq k_1 \ell < N$, and so each power $\zeta^{k_1 \ell}$ can be computed in
time $O(\phi(N)^2 \log N)$. Computing all the matrices $\pr_{\zeta^{k_1}, 1, \dots, 1}(\hat S)$ thus takes time $O(m \sigma_0(N) \phi(N)^2 \log N)$ where
$\sigma_0(N)$ is the number of divisors of $N$.
\item $\zeta^k$ for all $0 \leq k < D$. This can be done in time $O(D \phi(N)^2)$.
\end{itemize}

\noindent \textbf{II.d.i} We progress through the $d$-tuples $(k_1, \dots, k_d)$ in lexicographical order. Therefore each matrix was either precomputed or
can be obtained from the previous by multiplying half the entries in each row by some $\zeta^k$ for $0 \leq k < D$. Since the $\zeta^k$ were preprocessed,
this takes time at most $O(m \phi(N)^2)$ for each torsion point.\\

\noindent \textbf{II.d.ii} Computing the rank requires at most $O(m^3)$ multiplications in the field $\Q(\zeta)$ and at most $m^3$ additions. Thus computing
the rank for each torsion point takes time $O(m^3 \phi(N)^2)$.\\

\noindent \textbf{II.d total} The steps II.d.i and II.d.ii must be performed $\sigma_0(N) N^{d-1}$ times so they alone
require time $O(m^3 \sigma_0(N) N^{d-1} \phi(N)^2)$. This dominates the first prepocessing step so altogether the running time is
$O( (m^3 \sigma_0(N) N^{d-1} +D) \phi(N)^2 )$. Recall that it was checked in I.a that $\phi(N) \leq C_d D N^{(d-1)/d} $ and using the
(significant) overestimate $\sigma_0(N) < N$ we obtain a upper bound on running time of $O( (m^3 N^d + D) D N^{(2d-2)/d})$.\\

\noindent \textbf{Total running time} It is easy to see that step II.d dominates all other running times. Since it must be done for each positive $N$ up to
$N_0$, the running time for the algorithm is $O( m^3 D N_0^{d+1 + 2 \frac{d-1}{d}} + D^2 N_0^{2 \frac{d-1}{d} + 1}) = O(m^3 ( D N_0^{d+1 + 2 \frac{d-1}{d}})) =
O(m^3 D^{d^2 + 3d-1} (\log D)^{d^2 + 3d-2})$.

\paragraph{Configurations with coefficients in number fields}
We leave it for the reader to extend the above algorithm to arbitrary coefficient fields $K$. However, we remark that in the case of number fields,
the only changes are that higher order torsion points may need to be checked (\corref{esttwo}) and rank computations require
multiplications in $K(\zeta)$. The latter requires finding minimal polynomials of $\zeta$ over $K$ or equivalently factoring cyclotomic polynomials
over $K$, and that can be done via the algorithm in e.g. \cite{R04}.

\paragraph{Alternative computational methods}
The above algorithm is an exact algorithm guaranteed to work. However, performing exact calculations in $\Q(\zeta)$ does impose
some computational cost. One can also approximate $\zeta$ numerically and attempt to determine rank in which case
step II.c and the preprocessing in step II.d can be avoided and steps II.d.i and II.d.ii can be completed in time $O(m^3)$.
Consequently a numerical algorithm will run in time $O(m^3 N_0^{d+1}) = O(m^3 (D \log D)^{d^2+d})$ There is, however, no guarantee of correctness
without some a priori guarantee on the accuracy of the rank computations.

Another approach to speeding up rank computations is to work ``mod $p$'', i.e. reduce matrix entries to the finite field $\F_p(\zeta)$. There,
according to e.g. \cite{GaoPanario}, multiplication of elements can be computed in time $O(n \log n \log \log n)$. Yet another possibility is that one may compute the minors at the beginning,
and then determine if they evaluate to $0$ at torsion points using the algorithm in \cite{ChengVyalvi}.

\paragraph{An optimal $C_d$}
As the reader may notice, the constant $C_d$ grows rather quickly with dimension. Moreover, the impact on computation time is roughly a factor
of $C_d^{d^2}$ which can be significant even for small $d$. While we have given some thought to optimizing $C_d$, it would not be surprising
if an improvement could be made, and we do not know if $C_d$ is optimal for \lemref{deeptoshallow} and \lemref{deeptoshallowgeneral}
even in any asymptotic sense.

\section{Combinatorial results} \seclab{proof}
In this section, we prove Theorems \theorefX{dtwocomb} and \theorefX{dtwocombflfa}. All the
required definitions are given in this section.
The key ingredients are a linear representation of the
\Gamd{d} matroid (defined below in  \secref{linrepgamd}) and a theorem on
direction networks from
\cite{MT13}.

\subsection{Combinatorial types of colored graphs} \seclab{graphmats}
To describe our combinatorial classes of colored graphs,
we must understand the group associated to a colored graph.
We recall the construction only in the case of $\Gamma$ abelian
although it can be generalized to arbitrary groups. See e.g. \cite{MT13b}.
Suppose $(G, \bgamma)$ is a graph colored by an abelian group $\Gamma$.
For any oriented cycle $C$ of $G$, say $C \pcross ij$ if $C$ crosses $ij$ in the same orientation and
$C \ncross ij$ otherwise, and moreover set
$$\rho(C) = \sum_{C \pcross ij} \gamma_{ij} - \sum_{C \ncross ij} \gamma_{ij}.$$
We can extend $\rho$ uniquely to a map $\HH_1(G, \Z) \to \Gamma$. By abuse of notation, we will denote the image $\rho(G)$; this
is the group associated to the colored graph $(G, \bgamma)$.

Let $(G, \bgamma)$ be a $\Z^2$-colored graph with $m$ edges and $n$ vertices. Then, $(G, \bgamma)$
is {\it colored-Laman} if
\begin{itemize}
\item $m = 2n+1$
\item for any subgraph $G'$ on $n'$ vertices, $m'$ edges, and $c'$ components,
$m' \leq 2n' + 2\rk(\rho(G')) - 2c' -1$.
\end{itemize}
The colored graph $(G, \bgamma)$ is {\it colored-Laman-sparse} if it satisfies only the
inequality or, equivalently, is a subgraph
of a colored-Laman graph. A colored graph $(G, \bgamma)$ is a {\it colored-Laman circuit}
if it is an edge-wise minimal violation
of the above condition. A colored graph $(G, \bgamma)$ is {\it colored-Laman-spanning} if
it contains a vertex-spanning colored-Laman subgraph.

We say $(G, \bgamma)$ is a {\it Ross graph} if
\begin{itemize}
\item $G$ is a $(2,2)$-graph,
\item any subgraph $G'$ on $n'$ vertices and $m' > 2n'-3$ edges satisfies $\rho(G') \neq 0$.
\end{itemize}
Recall that, in general, $G$ is a $(k, \ell)$ graph if $m = kn - \ell$ and $m' \leq kn'-\ell$ for all subgraphs $G' \subset G$.
In particular, a {\it Laman graph} is a $(2,3)$-graph. Note that ``circuit,'' ``-spanning,'' and ``-sparse'' are similarly defined
for Ross graphs and $(k,\ell)$ graphs. We say $(G, \bgamma)$ is a {\it unit-area-Laman graph} if $m = 2n$, it is colored-Laman-sparse, and any subgraph $G' \subset G$
with $\rk(\rho(G')) = 2$ satisfies the strict inequality $m' < 2n' + 2\rk(\rho(G')) - 2c' -1$.

Recall that a {\it map-graph} is a graph where
each connected component has exactly one cycle. In particular, map-graphs have $m = n$ edges.
A $\Gamma$-colored graph $(G, \bgamma)$ is {\it \Gamd{1} } if it is a map-graph such that $\rho(C) \neq 0$ for each
cycle $C$ in $G$. (The collection of \Gamd{1} graphs is sometimes also called a frame matroid.)
We say that a $\Gamma$-colored graph $(G, \bgamma)$ is {\it \Gamd{d}}
if it is the edge-disjoint union of $d$ spanning \Gamd{1} graphs.

\begin{remark}
For $d = 1,2$ and $\Gamma$ finite cyclic, the set of \Gamd{d} graphs is the same as the set of cone-$(d,d)$ graphs in \cite{MT13a}.
\end{remark}

\Gamd{d} graphs can also be characterized by sparsity counts. For a connected $\Z/N\Z$-colored graph $(G, \bgamma)$
with a unique cycle $C$, we set $T(G) = 1$ if $\rho(C) = 0$ and $T(G) = 0$ if $\rho(C) \neq 0$. By \cite{MT13a}, a graph
$(G, \bgamma)$ on $m$ edges and $n$ vertices is \Gamd{1} if and only if
\begin{itemize}
\item $m = n$
\item for all subgraphs $G'$ on $m'$ edges and $n'$ vertices,
$$m' \leq n' -  \sum_{\substack{ \text{connected} \\\text{components } G_i \subset G}} T(G_i).$$
\end{itemize}
Using Edmonds' theorem on matroid unions \cite{ER66,E65b},
we can characterize \Gamd{d} graphs as follows.
\begin{lemma} \lemlab{sparsecond}
A $\Gamma$-colored graph $(G, \bgamma)$ on $m$ edges and $n$ vertices is \Gamd{d} if and only if
\begin{itemize}
\item $m = dn$
\item for all subgraphs $G'$ on $m'$ edges and $n'$ vertices,
$$m' \leq dn' - d \sum_{\substack{ \text{connected} \\\text{components } G_i \subset G}} T(G_i).$$
\end{itemize}
\end{lemma}

\subsection{Linear representations of the \Gamd{d} matroid} \seclab{linrepgamd}
Let $\Gamma = \Z/N\Z$, and let $(G, \bgamma)$ be a $\Gamma$-colored graph. Over all
edges $ij \in E(G)$, let $\vec v_{ij} = (a_{ij}^1, \dots, a_{ij}^d)$ where all $a_{ij}^k$ are algebraically independent
elements in some field extension of $\Q$. Let $\zeta$ be a primitive $N$th root of unity. Then, we define
$\vec M_{N, d,d}(G)$ to be the matrix with one row for each edge $ij$ as follows:
$$\begin{array}{cccccccc}      &   &         &  i                  &           & j                                                  &          & \\
ij & ( & \dots & -\vec v_{ij} & \dots & \zeta^{\gamma_{ij}} \vec v_{ij} & \dots  & )
\end{array} $$
\begin{remark}
Note that $\vec M_{N,d,d}(G)$ depends also on the choice of $\zeta$. However, \lemref{linrepgamd} below
holds for all choices.
\end{remark}

The key lemma is the following which is a special case of \cite[Corollary 5.5]{TaniGain}.
\begin{lemma} \lemlab{linrepgamd}
A $\Z/N\Z$-colored graph $(G, \bgamma)$ with $m = dn$ edges is \Gamd{d} if and only if
$\vec M_{N, d,d}(G)$ has rank $dn$.
\end{lemma}
\begin{proof}
This is a straightforward reinterpretation of Corollary 5.5 of \cite{TaniGain}. Note that in the notation of that paper
$\F = \C$ and $\rho: \Z/N\Z \to \GL(\C^d)$ is the map $\gamma \mapsto \zeta^{\gamma} \Id$. Moreover, the vectors
$x_{e, \psi}$ are precisely the rows of $\vec M_{N,d,d}(G)$.
\end{proof}

\subsection{Rank-preserving color changes}
Recall that the transition from the infinite graph $(\tilde G, \varphi)$ to a colored quotient graph $(G, \bgamma)$ requires
a choice of representative vertex for each $\Z^d$ vertex orbit in $\tilde G$. Changing the representative can result in a
change of the edge colors. For a given realization $\tilde G(\vec p, \vec L)$, such a change will alter the rigidity matrix, but since
ultrarigidity is a function only of the framework, the dimension of $\Lambda$-respecting motions is unchanged. We can, however,
describe such color changes without any reference to $\tilde G$. For any $(G, \bgamma)$, we say $(G', \bgamma')$ is an {\it elementary
valid color change} of $(G, \bgamma)$ if $G = G'$ as graphs and there is a vertex $k$ and $\gamma \in \Gamma$ such that
\begin{itemize}
\item[(1)] $\gamma_{ij}' = \gamma_{ij}$ if $i \neq k \neq j$
\item[(2)] $\gamma_{ik}' = \gamma_{ik} \gamma^{-1}$ for all (oriented) edges $ik$
\item[(3)] $\gamma_{kj}' = \gamma \gamma_{kj}$ for all (oriented) edges $kj$
\item[(4)] $\gamma_{kk}' = \gamma_{kk} $ for all loops $kk$
\end{itemize}
(Note that the analogous condition to (4) when $\Gamma$ is nonabelian is $\gamma_{kk}' = \gamma \gamma_{kk} \gamma^{-1}$.)
We say $(G', \bgamma')$ is a {\it valid color change} of $(G, \gamma)$ if it can be obtained from $(G, \bgamma)$ by a sequence of
elementary valid color changes.

\begin{lemma} \lemlab{colorchangefromcover}
Suppose $(G, \bgamma)$ and $(G', \bgamma')$ are two colored quotient graphs associated to the same infinite graph $(\tilde G, \varphi)$.
Then $(G', \bgamma')$ is a valid color change of $(G, \bgamma)$.
\end{lemma}
\begin{proof} The only difference arises from choices of vertex representatives. The effect of changing one vertex representative has exactly
the effect of an elementary change.
\end{proof}

While this easily implies the rigidity matrices for each colored graph have equivalent kernels, we want to find the same equivalence for slightly
more general matrices. In the rigidity matrix, the vectors $\vec d_{ij}$ must arise from some framework and are not completely arbitrary. We analyze
the kernels of the matrices $S_{G, \vec d}$ and $\hat{S}_{G, \vec d}$ for \emph{arbitrary} vectors $\vec d_{ij}$. We view the latter matrix
as a $\R[\Gamma]$-linear map $\mc X^{dn} \to \mc X^m$.

\begin{lemma} \lemlab{colchangepresrank}
Let $\vec d_{ij} \in \R^d$ be arbitrary vectors and let $(G, \bgamma)$ be a $\Z^d$-colored graph. If $(G', \bgamma')$ is a valid color change
of $(G, \bgamma)$, then $\ker(S_{G, \vec d}) \cong \ker(S_{G', \vec d})$ and for all finite index $\Lambda < \Gamma$
$$\ker(\hat{S}_{G, \vec d}) \cap \mc X_{\Lambda}^{dn}   \cong \ker(\hat{S}_{G', \vec d})  \cap \mc X_{\Lambda}^{dn} .$$
\end{lemma}
\begin{proof}
It suffices to prove lemma for elementary changes. Suppose the change is by $\gamma$ at vertex $k$. The kernel of $S_{G, \vec d}$ is equivalent to the set of vectors
$(\vec w, \vec M) \in \R^{dn} \times \Hom(\Gamma, \R^d)$ satisfying for all edges $ij$
$$ \langle \vec w_j, \vec d_{ij} \rangle - \langle \vec w_i, \vec d_{ij} \rangle + \langle \vec M(\gamma_{ij}), \vec d_{ij} \rangle = 0.$$
The kernel of $S_{G, \vec d}$ is the set of vectors satisfying
$$\begin{array}{rcll}  \langle \vec w_j, \vec d_{ij} \rangle - \langle \vec w_i, \vec d_{ij}  \rangle + \langle \vec M(\gamma_{ij}), \vec d_{ij} \rangle& = & 0
& \text{ if } i \neq k \neq j \text{ or } i = k = j\\
\langle \vec w_j, \vec d_{ij} \rangle - \langle \vec w_i, \vec d_{ij} \rangle + \langle \vec M(\gamma \gamma_{ij}), \vec d_{ij} \rangle &=& 0
& \text{ if } i = k \neq j \\
\langle \vec w_j, \vec d_{ij} \rangle - \langle \vec w_i, \vec d_{ij} \rangle + \langle \vec M(\gamma_{ij} \gamma^{-1}), \vec d_{ij} \rangle &=& 0
& \text{ if } i \neq k = j \end{array}  $$
The map $(\vec w, \vec M) \mapsto ( (\vec w_1, \dots, \vec w_{k-1}, \vec w_k + \vec M(\gamma), \vec w_{k+1}, \dots, \vec w_n) , \vec M)$ provides
the isomorphism $\ker(S_{G, \vec d}) \cong \ker(S_{G', \vec d})$.

The kernel of $\hat{S}_{G, \vec d}$ is equivalent to the set of vectors $\vec w \in \mc X^{dn}$ satisfying for all $ij$
$$ [ \vec w_j, \vec d_{ij} \otimes \gamma_{ij}^{-1}]  - [ \vec w_i, \vec d_{ij} \otimes 1] = 0$$
The map $\vec w \mapsto ( \vec w_1, \dots, \vec w_{k-1}, \vec w_k \gamma^{-1} , \vec w_{k+1}, \dots, \vec w_n)$ provides the isomorphism
$\ker(\hat{S}_{G, \vec d}) \cap \mc X_{\Lambda}^{dn}   \cong \ker(\hat{S}_{G', \vec d})  \cap \mc X_{\Lambda}^{dn}$.
\end{proof}

\subsection{A previous result on direction networks} \seclab{perdirnet}
A key ingredient in the proof is the ability to choose generic directions for the edges $\vec d_{ij} = \vec p_j + \vec L \gamma_{ij} - \vec p_i$.
More precisely we have the following theorem which is one direction of \cite[Theorem B]{MT13}.
\begin{prop} \proplab{dirnet}
Let $(G, \bgamma)$ be a $\Z^2$-colored graph which is colored-Laman. Then, there is a proper subvariety $V \subset \R^{2m}$
defined over $\Q$
such that if $\vec d = (\vec d_{ij}) \notin V$, then there exists a framework $G(\vec p, \vec L)$ and scalars $c_{ij} \neq 0$ satisfying
$c_{ij} \vec d_{ij} = \vec p_j + \vec L \gamma_{ij} - \vec p_i$ for all edges $ij \in E(G)$.
\end{prop}

\subsection{Proof of \theoref{dtwocomb}}
We begin by proving necessity. Suppose $G(\vec p, \vec L)$ is infinitesimally ultrarigid. Then, by
\corref{rigiditymatrix},
$S$ has rank
$2n+1$ and $pr_{\bomega}(\hat S)$ has $\C$-rank $2n$ for all torsion points $\bomega \neq (1,1)$. Thus, by \cite[Theorem A]{MT13},
$(G, \bgamma)$ is colored-Laman. Let $\Psi: \Z^2 \to \Z/N\Z$ be some surjective homomorphism. Let $\zeta$ be a primitive $N$th root of unity and let
$\bomega = (\zeta^{\Psi(e_1)}, \zeta^{\Psi(e_2)})$. Then, $pr_{\bomega}: \R[\Z^2] \to \F_{\omega}$ restricted to $\Z^2 \to \langle \zeta \rangle \cong \Z/N\Z$ is equivalent to $\Psi$.
It is clear from inspection that $pr_{\bomega}(\hat S)$ is a specialization of the matrix $\vec M_{N,2,2}(G, \Psi(\bgamma))$. By  \lemref{linrepgamd},
it follows that $(G, \Psi(\bgamma))$ is \Gamd{2}-spanning.

We now prove sufficiency. Choose $a_{ij}, b_{ij} \in \R$ for all edges which are algebraically independent over $\Q$.
Necessarily, the $\vec d_{ij} = (a_{ij}, b_{ij})$ avoid the subvariety $V$ as in \propref{dirnet}. By that same \propref{dirnet}, there is a framework
$G(\vec p, \vec L)$ such that $c_{ij} \vec d_{ij} = \vec p_j + \vec L\gamma_{ij} - \vec p_i$ for $c_{ij} \neq 0$. We can thus rescale each row $ij$ of
$\hat S$ by $1/c_{ij}$ to obtain a matrix $\hat S'$ with rows:
$$\begin{array}{cccccccc}      &   &         &  i                  &           & j                                                  &          & \\
ij & ( & \dots & -\vec d_{ij} & \dots & \vec d_{ij} \otimes [\gamma_{ij}] & \dots  & )
\end{array} $$
Clearly, $pr_{\bomega}(\hat S)$ has rank $2n$ if and only if $pr_{\bomega}(\hat S')$ does.
Using similar arguments to the above (but in reverse), $pr_{\bomega}(\hat S')$ is $\vec M_{N, 2,2}(G, \Psi(\bgamma))$ for some $N$ and epimorphism
$\Psi: \Z^2 \to \Z/N\Z$. By  \lemref{linrepgamd}, $pr_{\bomega}(\hat S')$ has rank $2n$. Moreover, by \cite[Theorem A]{MT13}, $S$ has rank $2n+1$.

We now prove the claim that we only need to verify that $(G, \Psi(\bgamma))$ is $\Delta$-$(2,2)$ for finitely many $\Psi$. Let $G(\vec p, \vec L)$ be as above
where the coordinates of the $\vec d_{ij}$ are generic. Let $q_1, \dots, q_k$ be all the $m \times m$ minors of the rigidity matrix $\hat S$.
Let $\hat C = C_2 D$ where $C_2$ is the constant from \secref{tpiv} and $D = \sum_{ij \in E(G)} \|\gamma_{ij}\|_1$, and set
$N_0 = \max(8500, (\hat C \log \hat C)^2)$.
Genericity of the $\vec d_{ij}$ implies the coefficient field $K$ is a purely transcendental extension of $\Q$ and so $K \cap \Q_{ab} = \Q$.
\lemref{linrepgamd} implies that all $q_i$ do not vanish at any torsion point $\bomega \neq \vec 1$ up to order $N_0$.
Since $[K \cap \Q_{ab}: \Q] = 1$ and $\widetilde{\deg}(q_i) \leq D$ for all $i$, by \propref{simplealggeneral},
the only torsion point in the variety defined by the $q_i$ is $\vec 1$. Consequently, $\tilde{G}(\vec p, \vec L)$
is infinitesimally ultrarigid.
\eop

Note that the ``Maxwell'' direction in the above proof applies mutatis mutandis to all dimensions regardless of the number of edges. We thus have the following necessary conditions
for infinitesimal ultrarigidity in all dimensions.

\begin{cor}
Let $(G, \bgamma)$ be a $\Z^d$-colored graph. If $G(\vec p, \vec L)$ is infinitesimally ultrarigid for some framework $(\vec p, \vec L)$,
then for all surjective homomorphisms $\Psi: \Z^d \to \Z/N\Z$, the graph $(G, \Psi(\bgamma))$ is \Gamd{d}-spanning.
\end{cor}

Moreover the proof implies the following effective version of \theoref{dtwocomb}.
\begin{cor}\corlab{effdtwocomb}
Let $\tilde{G}(\vec p,\vec L)$ be a generic $2$-dimensional periodic framework with
associated colored graph $(G,\bgamma)$ on $n$ vertices and $m=2n+1$ edges.
Let $D = \sum_{ij \in E(G)} \|\gamma_{ij}\|_1$ and $\hat C = C_2 D$.  Then $\tilde{G}(\vec p,\vec L)$
is infinitesimally ultrarigid if and only if $(G, \bgamma)$ is colored-Laman and $(G, \Psi(\bgamma))$ is $\Z/N\Z$-$(2,2)$ spanning for all
$N \leq \max(8500, (\hat C \log \hat C)^2)$ and epimorphisms $\Psi: \Z^2 \to \Z/N\Z$.
\end{cor}

\subsection{Relations between combinatorial classes}
Here, we state some basic relations among our combinatorial classes which will be useful for proving \theoref{dtwocombflfa} and presenting
our polynomial time combinatorial algorithms for checking the conditions therein.

\begin{lemma} \lemlab{altform}
A $\Z^2$-colored graph $(G, \bgamma)$ is $\Delta$-$(2,2)$ for every epimorphism $\psi: \Z^2 \to \Delta$ to finite cyclic $\Delta$ if and only if
every $\rho(G') = \Z^2$ for every $(2,2)$-circuit $G' \subset G$.
\end{lemma}

\begin{proof}
Assume the latter condition. Any $\Delta$-$(2,2)$ circuit of $(G, \psi(\bgamma))$ contains a $(2,2)$ circuit $G'$ for which, by assumption, $\rho(G') = \Z^2$.
Thus, there is no $\Delta$-$(2,2)$ circuit.

Assume the former condition. For any $(2,2)$ circuit $G' \subset G$, we must have $\psi(\rho(G')) \neq 0$ for all surjective representations $\psi: \Z^2 \to \Delta$ for
$\Delta$ cyclic. This implies $\rho(G') = \Z^2$.
\end{proof}

\begin{lemma} \lemlab{containsRoss}
All unit-area-Laman and colored-Laman graphs contain a spanning Ross graph.
\end{lemma}

\begin{proof}
Let $(G, \bgamma)$ be such a graph and choose a generic realization which is then necessarily infinitesimally rigid (in the forced symmetry sense).
If we impose the additional constraint that the lattice be fixed, then $(G, \bgamma)$ is rigid as a graph with fixed lattice. Since it is generic, it is infinitesimally
rigid as a fixed-lattice framework and hence contains a spanning Ross graph by \cite{R11} or \cite[Proposition 4]{MT13}.
\end{proof}

The next lemma establishes the equivalence of (iii) and (iv) of \theoref{dtwocombflfa}.
\begin{lemma}\lemlab{UALandRP2}
A colored graph $(G, \bgamma)$ is unit-area-Laman if and only if $(G, \bgamma)$ is colored-Laman-sparse and a Ross graph
plus $2$ edges.
\end{lemma}
\begin{proof}
The first implication is clear from \lemref{containsRoss}. Suppose $(G, \bgamma)$ satisfies the
latter condition. Since the graph is colored-Laman sparse and has $m=2n$ edges, the only way in which
it can fail to be unit-area-Laman is if there is a subgraph $G' \subset G$ with $\rk(\rho(G')) =2$ and
$m' = 2n' + 4 - 2c' -1 = 2n' + 3 - 2c'$. However, $G$ is a $(2,2)$-graph plus $2$ edges,
so for any subgraph $m' \leq 2n' + 2 - 2c'$.
\end{proof}
For algorithmic purposes, the following alternate characterization is more useful.
\begin{lemma} \lemlab{anotherchar}
A $\Z^2$-colored graph $(G, \bgamma)$ satisfies conditions (iii) and (iv)
of \theoref{dtwocombflfa} if and only if
$(G, \bgamma)$ is a Ross graph plus $2$ edges satisfying:
\begin{itemize}
\item[(a)] $\rho(G') = \Z^2$ for every $(2,2)$-circuit $G' \subset G$
\item[(b)] $\rho(G') \neq 0$ for every $(2,3)$-circuit $G' \subset G$
\end{itemize}
\end{lemma}
\begin{proof}
Assume that $(G,\bgamma)$ satisfies (iii) and (iv) from \theoref{dtwocombflfa}.
By \lemref{altform}, condition (a) holds. Condition (b) holds because $G$ is colored-Laman-sparse.

Now assume that (a) and (b) hold.
From condition (a), it is obvious that $(G, \psi(\bgamma))$ is
$\Delta$-$(2,2)$ for every surjective representation $\psi: \Z^2 \to \Delta$. What is left to do, by
\lemref{UALandRP2}, is show that $G$ is colored-Laman-sparse. For a contradiction, we assume that
there is a colored-Laman circuit $G'$ in $G$.  Let $n'$ and $m'$ be the number of vertices and
edges in $G'$, $c'$ be the number of connected components
and $r = \rk(\rho(G'))$.  Since $G'$ is a colored-Laman circuit, we have $m' = 2n' +2r - 2c'$.

Now we analyze each possible value of $r$.  Condition (b) rules out $r = 0$, since
minimality of circuits forces $G'$ to be connected, and thus a $(2,3)$-circuit with trivial $\rho$-image.  This
would contradict (b).

If $r=1$, then each connected component $G''$ of $G'$ has $\rk(\rho(G'')) = 1$ by minimality of
circuits.  This means that if $G''$ has $n''$ vertices, it has at least $2n''-1$ edges and
thus contains a $(2,2)$-circuit $H$.  According to (a) $H$ has $\rho$-image all of $\Z^2$
which is impossible if $r = 1$.

Finally, for $r=2$, $m' = 2n' + 4 - 2c'$.  Because $G$ is a Ross graph plus $2$ edges,
$G'$ spans at most $2n' + 2 - 2c'$ edges, which is again a contradiction.
\end{proof}

\subsection{Proof of \theoref{dtwocombflfa}}
\lemref{UALandRP2} implies (iii) and (iv) are equivalent, and clearly (ii) implies (i). We will show (iv) $\Rightarrow$ (ii) and (i) $\Rightarrow$ (iii).

\paragraph{(iv) $\Rightarrow$ (ii):} We need to show that $pr_{\bomega}(\hat S)$ has the maximal possible rank for all $\bomega$
for some $(\vec p, \vec L)$.
Since $(G, \bgamma)$ is colored-Laman-sparse, we can, as before, choose some $\vec p, \vec L$ so that the edge vectors $\vec d_{ij}$ are generic.
The same argument for the flexible lattice case implies that $pr_{\bomega}(\hat S)$ is full rank for $\bomega \neq \vec 1$.

By \corref{fixvolmatrix}, it suffices to show that the system
defined by $S$ and $\tr(\vec L^{-1}\vec M ) = 0$ (viewing  $\vec L$ as
a $2 \times 2$ matrix) has rank $2n+1$. This follows from \cite[Theorem 4]{MT14}.

\paragraph{(i) $\Rightarrow$ (iii)/(iv)}
Suppose that $(G, \bgamma)$ is infinitesimally f.l. ultrarigid for some generic placement $\tilde{G}(\vec p, \vec L)$. Since there are exactly $m = 2n$ edges, by \lemref{linrepgamd},
$(G, \psi(\bgamma))$ is $\Delta$-$(2,2)$ for every finite cyclic $\Delta$ and epimorphism $\psi: \Z^2 \to \Delta$.
By \lemref{altform}, $\rho(G') = \Z^2$ for every $(2,2)$ circuit $G' \subset G$.
It also follows from
\cite{R11} or \cite[Proposition 4]{MT13} that $(G, \bgamma)$ must be Ross-spanning.

By \lemref{anotherchar}, it remains only to prove that $\rho(G') \neq 0$ for $(2,3)$ circuits. Suppose not, so $\rho(G') = 0$ for
some $(2,3)$ circuit. We will find a contradiction to the maximality of the rank of $pr_{\bomega}(\hat S)$.
We can perform valid color changes so that the edge colors on a spanning tree are $0$, and
since $\rho(G') = 0$, the colors of the other edges become $0$ as well. This does not change the rank of $pr_{\bomega}(\hat S)$, yet in an uncolored
graph $pr_{\bomega}(\hat S_{G', \vec d}) = pr_{\vec 1}(\hat S_{G', \vec d})$. Moreover, since the edges are uncolored, the edge vectors $\vec d_{ij}$
are precisely $\vec p_j(0) - \vec p_i(0)$. If we set $\hat{ \vec p}_i = \vec p_i(0)$ for all $i \in V(G)$, then $pr_{\vec 1}(\hat S_{G', \vec d})$ is precisely
the rigidity matrix for the finite framework $G'(\hat{\vec p})$. Since $G'$ is not $(2,3)$-sparse, there is a dependency by Laman's theorem. \\

\eop

\subsection{Fixed-lattice ultrarigidity for arbitrary nonsingular lattices}

\begin{cor} \corlab{affinechange}
Let $\tilde{G}(\vec p,\vec L)$ be a $2$-dimensional periodic framework where $\vec p$ is generic and $\vec L$ is any arbitrary
nonsingular matrix. Moreover, assume the associated colored graph $(G, \bgamma)$ has $n$ vertices and $m = 2n$ edges.
Then, $\tilde{G}(\vec p,\vec L)$ is infinitesimally f.l. ultrarigid if and only if
$(G, \bgamma)$ satisfies condition (iii) or (iv) of \theoref{dtwocombflfa}.
\end{cor}
\begin{proof}
Assume the latter and fix some $\vec L$. By \theoref{dtwocombflfa}, any generic $\tilde{G}(\vec p', \vec L')$ is infinitesimally f.l. ultrarigid.
However, infinitesimal f.l. ultrarigidity is invariant under affine transformations, so using a suitable transformation we find $\tilde{G}(\vec p, \vec L)$
is infinitesimally f.l. ultrarigid for some $\vec p$. This implies $\tilde{G}(\vec p, \vec L)$ is infinitesimally f.l. ultrarigid for any generic $\vec p$ as well.

If we assume the former holds, then moreover generic $\tilde{G}(\vec p', \vec L')$ are infinitesimally f.l. ultrarigid and so we are done by \theoref{dtwocombflfa}.
\end{proof}

\subsection{Combinatorial algorithms for generic rigidity}
\theoref{dtwocomb} and \theoref{dtwocombflfa} provide combinatorial conditions for infinitesimal ultrarigidity in the case of the minimum possible number of edge
orbits. In this section, we discuss algorithms for checking these conditions. \theoref{dtwocomb} and algorithms from \cite{MT13, MT13a} guarantee
that there is some finite time algorithm in the fully flexible case. We will see that \corref{esttwo} implies the algorithm runs in time polynomial in $m$
and sizes of the edge colors (and so is technically exponential time). In the fixed-lattice/fixed-area case, we will see that a truly polynomial time
algorithm is possible. We begin with a quick exposition of an algebraic algorithm on vectors in $\Z^2$.

\subsubsection{Algorithm for determining the index of $\Z^2$ subgroups} \seclab{eucalg}
We discuss an algorithm which solves the following problem. Given $m$ vectors in $\Z^2$, determine the index of the subgroup they generate. First,
we explain the case $m = 2$. If we have $\lambda_1, \lambda_2 \in \Z^2$, then we can add an integer multiple of one to the other without
affecting the subgroup that is generated. So if $\lambda_1 = (a, b), \lambda_2 = (c, d)$, we can do such operations (following the Euclidean algorithm) to obtain
two vectors $\lambda_1' = (a', b'), \lambda_2' = (0, d')$ where $a' = \gcd(a, c)$. The index is then $a' d'$ which is the determinant of the
$2 \times 2$ with rows $\lambda_1', \lambda_2'$. Note that $d'$ is no larger than $\max(a, b, c, d)^2$.
(We could, of course, just take the determinant at the beginning, but we will use this as a subroutine.)
Note that the Euclidean algorithm runs in time $O(\log^2 \min(a, c) \log \log \min(a, c))$, so that is the running time here as well.\\

\noindent \textbf{Steps in the algorithm for general $m$}

\noindent Suppose the original vectors are $\lambda_1, \dots, \lambda_m$. \\

\noindent \textbf{I} In order from $i = 2$ to $m$, replace $\lambda_1, \lambda_i$ with the vectors obtained from the procedure described above
so that $\lambda_i$ has first coordinate $0$.\\

\noindent \textbf{II} Now, the vectors $\lambda_2, \dots, \lambda_m$ are essentially integers so run the Euclidean algorithm to get $\lambda_2 = (0, t)$
and $\lambda_i = 0$ for $i > 3$.\\

\noindent \textbf{III} Compute the determinant of the matrix with rows given by the new $\lambda_1$ and $\lambda_2$. This is the index.

\paragraph{Correctness:} Each step does not change the subgroup generated by the $\lambda_i$, so the correctness is clear.

\paragraph{Running time:} Let $D$ be the maximum size of a coordinate in any $\lambda_i$ (at the beginning). Step I takes time at most $O(m \log^2 D \log \log D)$. After
the completion of step I, the nonzero coordinate in $\lambda_2$ has size no larger than $D^2$. Thus, step II takes time at most
$O(m \log^2 D^2 \log \log D^2) = O(m \log^2 D \log \log D)$.

\subsubsection{Combinatorial algorithm for fixed area/fixed lattice} \seclab{combalg}
We begin with a polynomial time algorithm for testing the combinatorial condition (iii) and (iv) of \theoref{dtwocombflfa}. As we will see, the correctness depends on a third characterization of (iii) and (iv).\\

\noindent \textbf{Steps in the algorithm} \\

\noindent \textbf{I} Check if $m = 2n$. Extract a spanning Ross subgraph $R$ if possible and stop if it is not.
This can be done with the algorithm from \cite{BHMT11}.\\

\noindent \textbf{II} For every pair of edges $ij, i'j' \in E(G)$, do the following for $G' = G  - \{ij, i'j'\}$:

\begin{itemize}
\item[(a)] Determine if $G'$ is a $(2,2)$-graph with the pebble game algorithm \cite{LS08}.
If it is not a $(2,2)$-graph, continue to the next pair of edges. Otherwise go to step II.b.
\item[(b)] Determine if $G'$ is a Ross graph. If it is, continue to II.c, and otherwise stop.
\item[(c)] For each of $ij, i'j'$ compute the $(2,2)$-circuit $C_{ij}, C_{i'j'}$
in $G' + ij, G'+i'j'$ respectively (again
using the pebble game \cite{LS08}). Check if $\rho(C_{ij}) = \Z^2 = \rho(C_{i'j'})$. If they are not all equal,
stop and otherwise continue to the next
pair of edges.

One way to check if $\rho(C_{ij}) = \Z^2$ is as follows. First, find a spanning tree $T \subset C_{ij}$
and fundamental cycles $B_1, \dots, B_k$.
Choose some base vertex $a_0 \in V(T)$, and for the unique path $P_{a_0 a}$ in $T$ from $a_0$ to a vertex
$a \in V(T)$, compute $\rho(P_{a_0 a})$, i.e.
the sum of edge colors on edges in the path. Then,
$\rho(B_\ell) = \rho(P_{a_0 i_\ell}) + \gamma_{i_\ell j_\ell} - \rho(P_{a_0 j_\ell})$ where
$B_\ell$ is the fundamental cycle for edge $i_\ell j_\ell \in E(C_{ij}) - E(T)$. Apply the algorithm
from \secref{eucalg} to the collection
$\rho(B_1), \dots, \rho(B_k)$. If the index is $1$, continue and otherwise stop.
\end{itemize}

\noindent \textbf{III} If the algorithm proceeded through all previous steps without stopping,
then the framework satisfies conditions (iii)/(iv) and otherwise not.

\paragraph{Correctness:}
We check that the algorithm verifies the conditions of \lemref{anotherchar}.
Step I verifies the graph is Ross plus $2$ edges.
It remains to show that conditions
(a) and (b) from \lemref{anotherchar} are also checked.

We start with (b). We may assume the algorithm passed step I, and so we know $G$ has a Ross spanning
subgraph and thus a $(2,2)$
spanning subgraph. Thus any $(2,3)$ circuit $G' \subset G$ necessarily extends to some $(2,2)$ basis
$B$ which is $G$ minus two edges.
Consequently, at some point the algorithm will check if $B$ is a Ross graph (assuming (a) and (b) are not previously
violated) and if it is, that is a
certificate that $G' \subset B$ has
nonzero $\rho$-image. If $B$ is not Ross, then some violation of (b) occurs and the algorithm stops.

Now consider (a). Again assume step I has completed. Let $G'$ be a $(2,2)$ circuit. By similar reasoning
as for (b), $G' - ij$
is a $(2,2)$ graph and hence part of a $(2,2)$ basis $B$ which is $G$ minus two edges. Necessarily $ij \notin E(B)$, so
$G'$ is the unique $(2,2)$ circuit in $B + ij$, and so step II.c will check if $\rho(G') = \Z^2$ or not. \eop

\paragraph{Running Time:}
We set $D = \sum_{ij \in E(G)} \|\gamma_{ij} \|_1$. The running times of each step are as follows:

\noindent \textbf{I}  The algorithm of \cite{BHMT11} runs in time $O(m^2)$ .\\

\noindent \textbf{II.a} For each $G'$, this takes time $O(m^2)$. \\

\noindent \textbf{II.b} Like step I, this takes $O(m^2)$.\\

\noindent \textbf{II.c} Computing the circuits takes time $O(m^2)$. (In fact, if one continues with the pebble game algorithm from II.b, this
can be done even faster.) Finding the maximal tree and $\rho(B_\ell)$ for all $\ell$ takes time $O(m)$. Since $D$ is larger than any coordinate
in any $\rho(B_\ell)$, checking if $\rho(G') = \Z^2$ takes time $O(m \log^2 D \log \log D)$. \\

\noindent \textbf{Total:} Since there are $m^2$ such $G'$ in step II, we get a total running time of $O(m^4 + m^3  \log^2 D \log \log D)$.

\subsubsection{Combinatorial algorithm for generic rigidity for flexible lattice}
In the case of the fully flexible lattice, we only know an algorithm which is polynomial in $m$ but only
polynomial in $D = \sum_{ij \in E(G)} \|\gamma_{ij} \|_1$,
not polylogarithmic as in the previous case. The main reason for this is that we know of no appropriate analogue to \lemref{altform} when $m = 2n+1$.
In this case, we will only verify that the algorithm is polynomial in $m, D$ and not give exact exponents.\\

\noindent \textbf{Steps in algorithm:} \\

\noindent \textbf{I} First, we verify the graph is colored-Laman via the algorithm as described in \cite{MT13}. \\

\noindent \textbf{II} Compute $\hat C = C_2 D$ where $C_2$ is the constant from \secref{tpiv}, and compute
$N_0 = \max(8500, (\hat C \log \hat C)^2)$. For every $N < N_0$ and surjective homomorphism $\Psi: \Z^2 \to \Z/N\Z$,
do the following.

\begin{itemize}
\item[(a)] Compute the $\Z/N\Z$-colored graph $(G, \psi(\bgamma))$ where colors are represented
by an integer in $0, \dots, N-1$.
\item[(b)] For each edge $ij$, test whether $(G - ij, \psi(\bgamma))$ is $\Z/N\Z$-$(2,2)$ using the
algorithm from \cite{MT13a} (where
such graphs are called ``cone-$(2,2)$''). If $(G - ij, \psi(\bgamma))$ is not $\Z/N\Z$-$(2,2)$ for
all $ij$, then stop and otherwise continue.
\end{itemize}

\noindent \textbf{III} If the algorithm never stopped at II.b, then the graph is generically rigid and otherwise not.

\paragraph{Correctness:} This follows directly from \corref{effdtwocomb}.

\paragraph{Running Time:}
Each of the algorithms cited from \cite{MT13} and \cite{MT13a} run in polynomial time in $m$. The number of $\Psi$ to check in step II is polynomial in $D$, so the total running time is polynomial in $m$ and $D$.

\section{Closing Remarks} \seclab{conc}

\subsection{Infinitesimal ultraflexibility versus ultraflexibility}
Just as with most contexts, infinitesimal (ultra)rigidity implies (ultra)rigidity.
Specifically, a framework which is infinitesimally ultrarigid will have
only trivial $\Lambda$-respecting rigid motions for all finite index $\Lambda < \Z^d$.
On the other hand, it does not follow obviously that if a generic framework
is infinitesimally ultraflexible, then it must necessarily have some finite $\Lambda$-respecting flex.
Even if it is generic from the viewpoint of
$\Z^2$-periodicity,  from the viewpoint of $\Lambda$-periodicity the framework is especially symmetric.
Indeed, there are colored graphs such that
all its generic realizations are infinitesimally f.l. infinitesimally ultraflexible and f.l. ultrarigid.
\figref{ufbr} shows two colored graphs that are generically infinitesimally f.l. infinitesimally ultraflexible but still
generically f.l. ultrarigid. In the case of the fixed-area and fully flexible lattice,
it is still an open question.
\begin{figure}[htbp]
\centering
\subfigure[]{\includegraphics[width=0.35\textwidth]{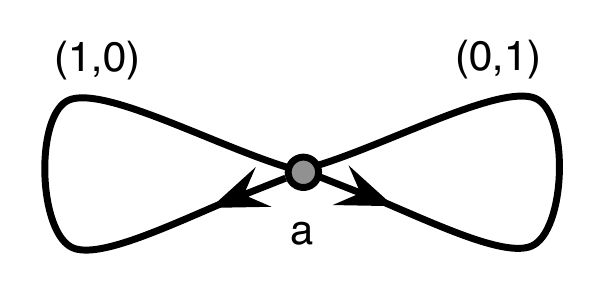}}
\subfigure[]{\includegraphics[width=0.35\textwidth]{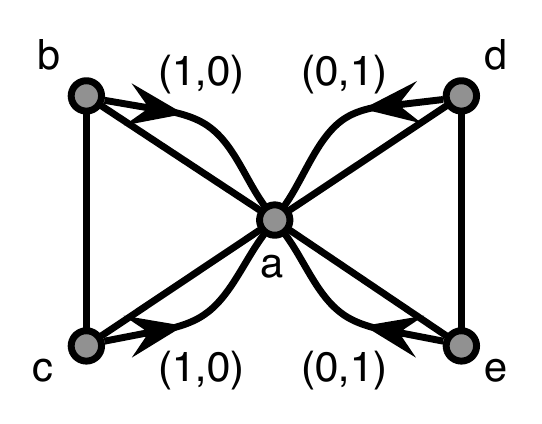}}
\caption{		}
\label{fig:ufbr}
\end{figure}
\begin{prop} \proplab{ultraflexbutrigid}
Let $(G, \bgamma)$ be as in Figure \ref{fig:ufbr}.(a) or Figure \ref{fig:ufbr}.(b). Any generic realization $G(\vec p, \vec L)$ is infinitesimally f.l. ultraflexible and f.l. ultrarigid.
\end{prop}
\begin{proof}
We begin with Figure (a).
First note that for $\Psi(\gamma_1, \gamma_2) = \gamma_2 \; (\text{mod} 2)$, the graph $(G, \Psi(\bgamma))$ is not $\Z/2\Z$-$(2,2)$, and so
generic realizations must be infinitesimally ultraflexible by \theoref{dtwocombflfa}. We fix now some arbitrary $\Lambda < \Z^2$ and prove that there
are only trivial $\Lambda$-respecting motions.

Let $\tilde G(\vec p, \vec L)$ be the realization of the corresponding infinite graph $\tilde G$. Let $a$ be the unique vertex of $G$.
Then, $\vec p_a(\gamma) = \vec p_a(0) + \vec L(\gamma)$ for all $\gamma \in \Z^2$. Let $e_1, e_2$ be the standard basis vectors of $\Z^2$
and let $t_1, t_2$ be the smallest positive integers satisfying $t_i e_i \in \Lambda$.
Any $\Lambda$-respecting motion must necessarily preserve the difference $\vec p_a( t_i e_i + \gamma) - \vec p_a( \gamma) = \vec L(t_i e_i)$
for all $\gamma \in \Gamma$. Since
$$\vec p_a( t_i e_i + \gamma) - \vec p_a( \gamma) = \sum_{k=1}^{t_i} ( \vec p_a(k e_i + \gamma) - \vec p_a(  (k-1) e_i + \gamma) = \sum_{k=1}^{t_i} \vec L(e_i),$$
the sequence of vertices is ``pulled tight'' and so any motion must preserve the difference $\vec p_a( k e_i + \gamma) - \vec p_a( (k-1) e_i + \gamma) = \vec L(e_i)$ for $1 \leq k \leq t_i$.
This implies that the difference $\vec p_a(e_i + \gamma) - \vec p_a(\gamma)$ for any $\gamma, i$ is the constant vector $\vec L(e_i)$ under any motion, i.e. all motions are trivial.

\begin{figure}[htbp]
\centering
\includegraphics[width=0.6\textwidth]{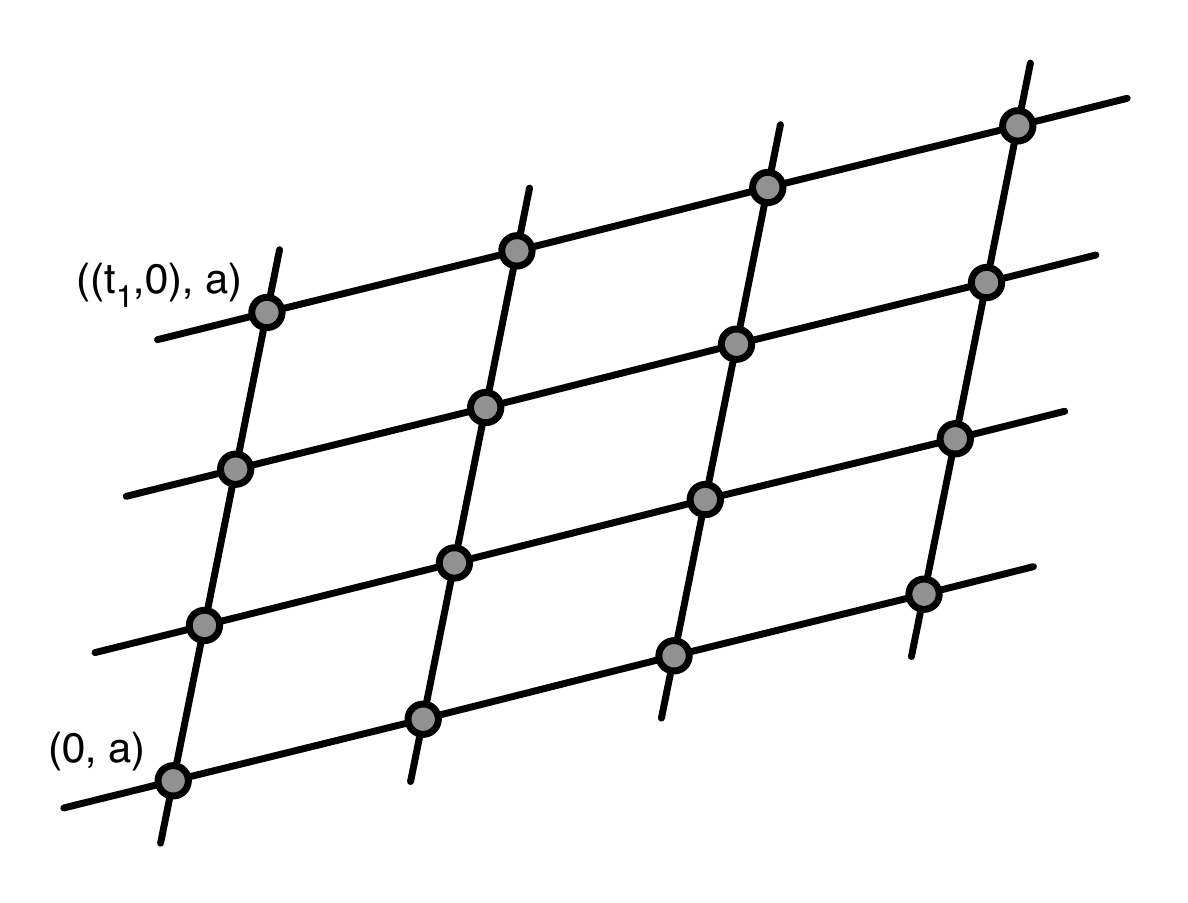}
\caption{Periodic realization of the graph in Figure \ref{fig:ufbr}.(a).}
\label{fig:pulltight}
\end{figure}

Now, let $(G, \bgamma)$ be the graph in figure (b). Let $\Psi$ be as above. Then,
$(G, \Psi(\bgamma))$ is not $\Delta$-$(2,2)$ since the graph spanned by vertices $a, b, c$ is $(2,1)$-tight but of trivial $\Delta$ color.
By \theoref{dtwocombflfa}, generic realizations are infinitesimally ultraflexible. However,
as Figure \ref{fig:ufbrreal} shows, the vertex $(a, \gamma)$ is connected to $(a, \gamma \pm e_i)$ for $i = 1,2$ by rigid graphs. Thus, as
in the previous example, regardless of $\Lambda$, the orbit of $(a, \gamma)$ is pulled tight and via similar arguments the framework is rigid.
\end{proof}

\begin{figure}[htbp]
\centering
\includegraphics[width=0.8\textwidth]{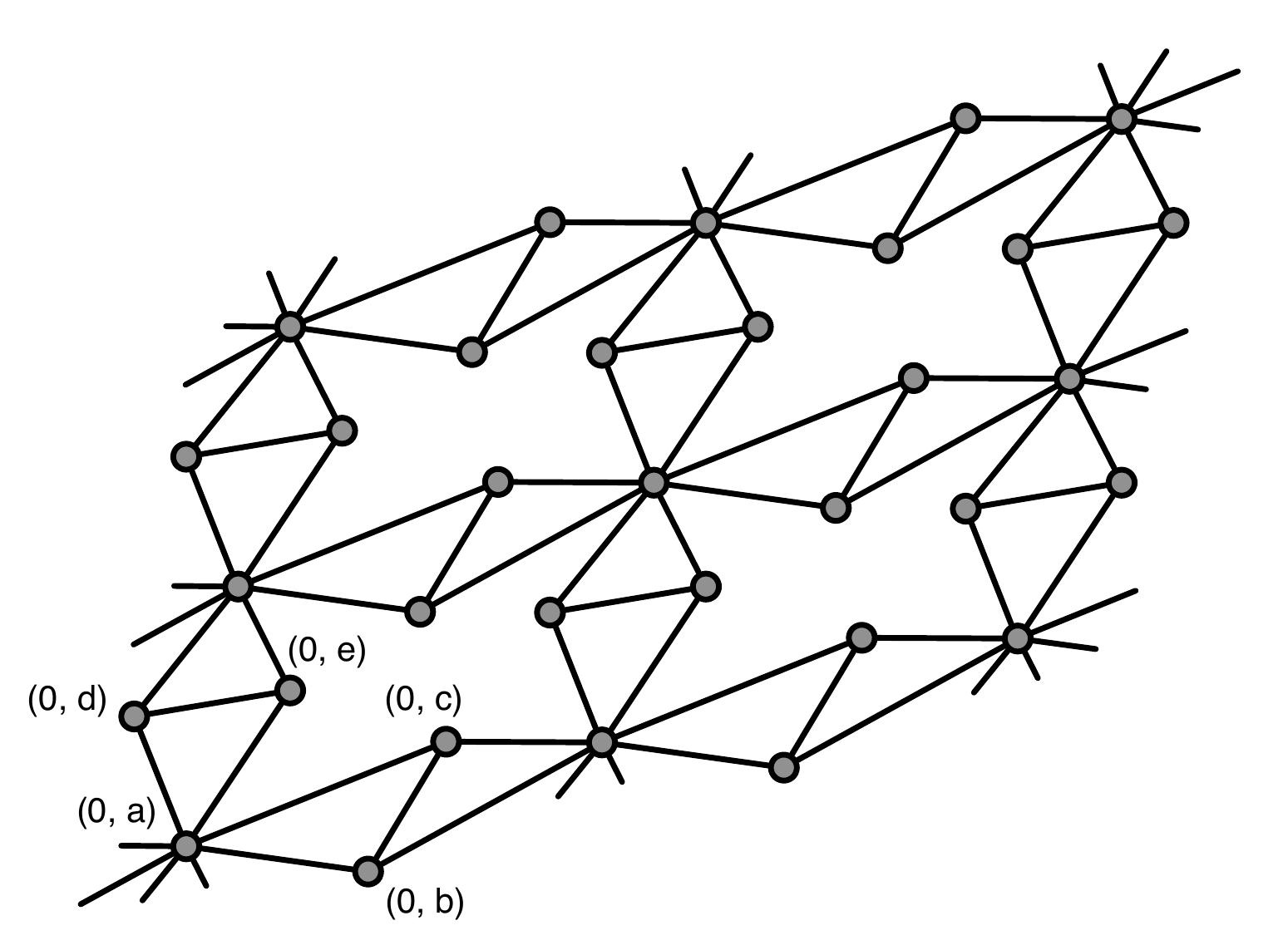}
\caption{Periodic realization of the graph in Figure \ref{fig:ufbr}.(b).}
\label{fig:ufbrreal}
\end{figure}

In light of the above examples, we ask the following.

\begin{problem} Characterize those graphs for which infinitesimal ultraflexibility implies ultraflexibility.
\end{problem}

\subsection{Some open questions:}
For many situations, infinitesimal rigidity is preserved under any sufficiently small deformation of a framework $G(\vec p)$
(not necessarily preserving lengths). The reason the property holds is that infinitesimal rigidity holds outside some
proper algebraic subvariety. However, the set of infinitesimally ultrarigid frameworks is, a priori, the complement of
\emph{infinitely} many subvarieties (one for each torsion point), and so it is unclear that the set is open.

\begin{question} For a given periodic graph, is the space of infinitesimally ultrarigid frameworks open? Does it contain any open sets?
\end{question}

The paper \cite{BS15} provides some evidence that the answer to the latter question is yes. In \cite{BS15}, it is shown that
periodic pointed pseudo-triangulations are f.a. infinitesimally ultrarigid and adding a single edge orbit produces an infinitesimally ultrarigid framework.
Since the property of being a periodic pointed pseudo-triangulation is preserved under small perturbations, this produces
open sets of ultrarigid frameworks.

On the other hand, in the context of fixed lattice ultrarigidity, Connelly--Shen--Smith have produced a
continuous $1$-parameter family of frameworks where both the infinitesimally ultrarigid and ultraflexible
frameworks are dense in the set of parameters. (See Theorem 9.1 of \cite{CSS14}. A more thorough
description of the family is available in the corresponding appendix.) In this context then,
the answer to the former question is, in general, negative. Moreover, it seems likely that this
example can be modified to apply to the fully flexible context. Thus, one preliminary project might
be to find a periodic graph where the infinitesimally ultrarigid realizations constitute an open set, if
indeed such a periodic graph exists.

The results of \cite{BS15}, \cite{Haasetal2005} and this paper lead to another natural question. In \cite{Haasetal2005},
it is shown that a planar Laman graph necessarily has a realization as a pointed pseudo-triangulation.
As was shown in \cite{BS15}, $m = 2n$ for a periodic pointed pseudo-triangulation, and so the frameworks must satisfy
the conditions of \theoref{dtwocombflfa}.

\begin{question} If a colored graph satisfies the conditions of \theoref{dtwocombflfa} and admits a planar periodic realization,
does it admit a realization as a periodic pointed pseudo-triangulation?
\end{question}

Our combinatorial theorems \theorefX{dtwocomb} and \theorefX{dtwocombflfa} characterize
generic infinitesimal ultrarigidity when the number of edges is the minimal possible. However,
infinitesimal ultrarigidity is not obviously matroidal (and almost certainly not) on colored graphs. Moreover,
for each torsion point $1 \neq \bomega \in \C^2$, we only understand generically the rank of $\pr_{\bomega}(\hat S_{G, \vec p, \vec L})$
when we assume additional combinatorial information about $(G, \bgamma)$, i.e. that it is colored-Laman-sparse.
Therefore, the following closely related problems remain open:

\begin{problem} In dimension $2$ (or higher),
give a complete combinatorial characterization of the linear matroid given by the generic rank of $\pr_{\bomega}(\hat S_{G, \vec p, \vec L})$.
\end{problem}

\begin{problem} Characterize, without any assumption on the number of edges, the generically infinitesimally ultrarigid graphs.
\end{problem}

\bibliographystyle{abbrvnat}

\end{document}